\documentclass[12pt]{article}
\usepackage{amsmath,amssymb,latexsym}
\usepackage{amscd,bm}% bm for bold symbols
\usepackage{color}
\usepackage{amsfonts}
\usepackage{wasysym}
\usepackage{mathrsfs,calligra}
\usepackage{hyperref}

\newtheorem{fact}{Fact}[section]
\newtheorem{theorem}[fact]{Theorem}
\newtheorem{corollary}[fact]{Corollary}

\newtheorem{lemma}[fact]{Lemma}

\newenvironment{proof}[1][Proof]{\textbf{#1.} }{\ \rule{0.5em}{0.5em}}

\numberwithin{equation}{section}
\newcommand{\eps}{{\varepsilon}}
\newcommand{\diag}{\operatorname{diag}}
\newcommand{\tr}{\operatorname{trace}}

\newcommand{\id}{\operatorname{id}}
\newcommand{\sdet}{\operatorname{sdet}}

\newcommand{\0}{{\mathbf{0}}}
\newcommand{\lhs}{{\rm lhs}}

\newcommand{\sk}[1]{\left\langle #1 \right\rangle}
\newcommand{\hor}{{\mathrm{hor}}}
\newcommand{\cart}{{\mathrm{cart}}}
\newcommand{\susy}{{\mathrm{susy}}}
\newcommand{\even}{{\mathrm{even}}}
\newcommand{\odd}{{\mathrm{odd}}}
\newcommand{\body}{{\mathrm{body}}}
\newcommand{\soul}{{\mathrm{soul}}}
\newcommand{\dach}{\mathsf{\Lambda}}

\newcommand{\W}{\mathcal{W}}
\newcommand{\M}{\mathcal{M}}
\newcommand{\I}{\mathcal{I}}
\newcommand{\B}{\mathcal{B}}
\newcommand{\U}{\mathcal{U}}
\newcommand{\V}{\mathcal{V}}
\newcommand{\E}{\mathbb{E}}

\newcommand{\cL}{\mathcal{L}}
\newcommand{\cR}{\mathcal{R}}
\newcommand{\s}{\mathscr{S}}
\newcommand{\sstar}{{\pmb{\mathscr{S}}}^*}

\newcommand{\matr}[1]{\left(\begin{array}#1\end{array}\right)}
\newcommand{\R}{{\mathbb R}}  
\newcommand{\N}{{\mathbb N}}

\newcommand{\C}{{\mathbb C}}  
\newcommand{\F}{{\mathcal{F}}}
\newcommand{\T}{{\mathcal{T}}}
\newcommand{\G}{{\mathcal{G}}}  
\newcommand{\bG}{{\boldsymbol{\mathcal{G}}}}  
\newcommand{\A}{{\mathcal{A}}}

\begin{document}
\thispagestyle{empty}

\begin{center}
{\LARGE Martingales and some generalizations 
arising from the supersymmetric
hyperbolic sigma model}\\[3mm]

{\large Margherita Disertori\footnote{Institute for Applied Mathematics
\& Hausdorff Center for Mathematics, 
University of Bonn, \\
Endenicher Allee 60,
D-53115 Bonn, Germany.
E-mail: disertori@iam.uni-bonn.de}
\hspace{1cm} 
Franz Merkl \footnote{Mathematisches Institut, Ludwig-Maximilians-Universit{\"a}t  M{\"u}nchen,
Theresienstr.\ 39,
D-80333 Munich,
Germany.
E-mail: merkl@math.lmu.de
}
\hspace{1cm} 
Silke W.W.\ Rolles\footnote{
Zentrum Mathematik, Bereich M5,
Technische Universit{\"{a}}t M{\"{u}}nchen,
Boltzmannstr.\ 3,
D-85748 Garching bei M{\"{u}}nchen,
Germany.
E-mail: srolles@ma.tum.de}
\\[3mm]
{\small \today}}\\[3mm]
\end{center}

\begin{abstract}
We introduce a family of real random variables $(\beta,\theta)$
arising from the supersymmetric nonlinear sigma model and
containing the family $\beta$ introduced by Sabot, Tarr\`es, and Zeng
\cite{sabot-tarres-zeng15} in the context of the vertex-reinforced jump process. 
Using this family we construct an exponential martingale
generalizing the one considered in \cite{disertori-merkl-rolles2017}. 
Moreover, using the full supersymmetric nonlinear sigma model 
we also construct a generalization of the  exponential martingale
involving Grassmann variables.\\[3mm]
\end{abstract}

\section{Introduction and main results}
\label{sect:intro}

The nonlinear supersymmetric hyperbolic sigma ($H^{2|2}$) model
was introduced by Zirnbauer in \cite{zirnbauer-91} as a toy model
for quantum diffusion. The corresponding measure can be better analyzed after 
passing to horospherical coordinates $(u,s)$ as in
\cite{spencer-zirnbauer2004} (for the nonsupersymmetric version)  and
$(u,s,\overline\psi,\psi)$ as in \cite{disertori-spencer-zirnbauer2010}
(cf. details below).
In particular a phase transition in dimension $d\geq 3$ was proved, see    
\cite{disertori-spencer-zirnbauer2010} and \cite{disertori-spencer2010}.

The $H^{2|2}$ model has an interpretation
as a random Schr\"odinger operator \cite{disertori-spencer2010}
and unexpectedly also as mixing measure for the vertex-reinforced jump process \cite{sabot-tarres2012}.
This process was conceived by Werner and first developed by  Davis and Volkov \cite{davis-volkov1}
\cite{davis-volkov2}.

More recently  Sabot, Tarr\`es, and Zeng developed further
the random Schr\"odinger operator interpretation
\cite{sabot-tarres-zeng15} \cite{sabot-zeng15}.
In particular they derived the explicit law for the random
potential, and constructed two families of martingales in discrete time.
One of them is the key ingredient to derive a characterization
of recurrence/transience behavior of the vertex-reinforced jump
process. In \cite{sabot-zeng17} Sabot and Zeng connected these families
to  certain continuous time martingales.

The above two families of discrete time martingales are only the first
instances of an infinite hierarchy of martingales described in
\cite{disertori-merkl-rolles2017}. All these martingales involve
only the $u$ components of the $H^{2|2}$ model.
In this paper we extend these martingales to even larger families
involving all the variables $(u,s,\overline\psi,\psi).$

\paragraph{How this article is organized.}
In  Sections \ref{sect:intro} and \ref{sect:2} we consider only the 
marginal $\mu^W(du\, ds)$ of the full $H^{2|2}$ model
obtained by integrating out the Grassmann variables $(\overline\psi,\psi).$
It is introduced in Section \ref{subsec:susy-model}.
The random variables $u$ encode the asymptotics of local
 times for a time changed vertex reinforced jump process while
 the random variables $s$ describe the corresponding fluctuations.
For details see \cite{merkl-rolles-tarres2017}.

In Section \ref{subsect:scaling}   we introduce a scaling transformation $\s$
for the variables $(u,s).$ The effect of this scaling on 
the measure $\mu^W$ is formulated in Theorem~\ref{thm:scaling}.  
This theorem is in turn the key ingredient to prove the martingale property,
which extends Theorem 2.6 and Corollary 2.7 from   
\cite{disertori-merkl-rolles2017} to test functions depending on $(u,s)$ variables.
Note that when the test function depends only on the $u$ variable, we recover the
martingales derived in \cite{disertori-merkl-rolles2017}.
The martingale property on an infinite graph for the marginal $\mu^W$
is stated in Section \ref{subsect:infinite}, while  Section \ref{subsect:finite}
contains some preliminary results in finite volume.
All these results are proved in Section \ref{sect:2}.

In Section \ref{sect:3}  we extend the results of Sections \ref{sect:intro}
and \ref{sect:2} to the full $H^{2|2}$
super-measure studied
in \cite{disertori-spencer-zirnbauer2010}, where Grassmann variables are included.
This requires also a generalization of the notion of martingale to a 'susy martingale', not to be confused
with the notion of supermartingale in standard probability. Here the test functions
may depend on Grassmann variables too. In particular when the test function depends only on the real variables
$u,s$ but not on the Grassmann variables, we recover the martingales described in 
Theorem~\ref{thm:generating-mg} and Corollary~\ref{eq:cor-derivatives-are-mgs}.

\subsection{The nonlinear supersymmetric hyperbolic sigma  model}
\label{subsec:susy-model}

Let $\tilde G=(\tilde V,\tilde E)$ be a finite connected graph with vertex set
$\tilde V$ and set of undirected edges $\tilde E$. We assume that $\tilde G$ 
has no direct loops and no parallel edges. 
We write $i\sim j$ if there is an edge between $i$ and $j$. 
Let $\delta\in\tilde V$
be a distinguished vertex and set $V=\tilde V\setminus\{\delta\}$. Every 
edge $(i\sim j)\in\tilde E$ gets a weight $W_{ij}=W_{ji}>0$. For convenience of 
notation, we set $W_{ij}=0$ for all $i,j\in\tilde V$ with $i\not\sim j$. 
The euclidean scalar product is denoted by 
$\langle a,b\rangle=\sum_{i\in I}a_ib_i$, where $I=V$ or $I=\tilde V$, depending 
on the type of $a$ and $b$. Let 
\begin{align}
\label{eq:def-Omega}
\Omega_V:= & 
\left\{ (u=(u_i)_{i\in \tilde V},s=(s_i)_{i\in \tilde V})\in\R^{\tilde V}\times\R^{\tilde V}:\; 
u_\delta=0,s_\delta=0\right\}.
\end{align}
We define the matrix $A^W(u)\in\R^{\tilde V\times\tilde V}$ by 
\begin{align}
\label{eq:def-A}
A_{ij}^W(u)=\left\{\begin{array}{ll}
-W_{ij}e^{u_i+u_j} & \text{for }i\neq j, \\
\sum_{k\in \tilde V} W_{ik}e^{u_i+u_k}  & \text{for }i=j.
\end{array}\right. 
\end{align}
Let $A^W_{VV}(u)$ denote its restriction to $V\times V$. We define 
$\rho^W:\Omega_V\to[0,\infty)$ by 
\begin{align}
\label{eq:def-rho}
\rho^W(u,s)= & e^{-\frac12 \sk{s,A^{W}(u)s}}\det A^W_{VV}(u) 
\prod_{(i\sim j)\in \tilde E} e^{-W_{ij}[\cosh(u_i-u_j)-1]} \\
= & \det A^W_{VV}(u) 
\prod_{(i\sim j)\in \tilde E} e^{-W_{ij}[\cosh(u_i-u_j)-1+\frac12(s_i-s_j)^2e^{u_i+u_j}]} .
\end{align}
Using the reference measure
\begin{align}
\label{eq:def-zeta}
\zeta(du_ids_i)=e^{-u_i}\, du_ids_i
\end{align}
on $\R^2$, the supersymmetric sigma model is described by the 
following probability measure on $\Omega_V$:
\begin{align}
\mu^W(du\, ds) =& \rho^W(u,s) \prod_{i\in V}\frac{1}{2\pi}\, \zeta(du_ids_i),
\label{eq:def-mu-W}
\end{align}
where we drop the Dirac measure located at $(u_\delta,s_\delta)=(0,0)$ 
in the notation. 
We denote the expectation with respect to $\mu^W$ by $\E_{\mu^W}$. 

\paragraph{Notation.}
In the following, operations are frequently
to be read componentwise, like $a^2+b^2=(a_i^2+b_i^2)_{i\in\tilde V}$, 
$e^{-u}b/a=(e^{-u_i}b_i/a_i)_{i\in\tilde V}$, 
$\log a=(\log a_i)_{i\in\tilde V}$. 

\subsection{Scaling}
\label{subsect:scaling}

In this section, we introduce the scaling transformation $\s.$
It arises naturally as action of a group $\G_{V}.$
The reader may skip this subsection except from the definitions
given in formulae \eqref{eq:def-scaling-s}  and
\eqref{eq:def-rescaled-weights}.

\paragraph{The group $\G$.}
For $(a,b)\in\R^+\times\R$, we abbreviate
\begin{align}
\label{eq:def-matrix-ab}
[a,b]:=
\left(\begin{array}{cc}
a & b\\
0 & 1 \end{array}\right). 
\end{align}
The set of matrices 
\begin{align}
\label{eq:def-G}
\G:=\{[a,b]:a>0,b\in\R\}
\end{align}
endowed with matrix multiplication forms a non-Abelian group. 
Its group operation can be written in the following form:
\begin{align}
\label{eq:group-op}
[a'',b'']=[a,b]\cdot[a',b']=[aa',b+ab']. 
\end{align}
The group $\G$ has the neutral element $[1,0]$; the inverse is given by 
\begin{align}
\label{eq:inverse}
[a,b]^{-1}=[1/a,-b/a]. 
\end{align}
We endow $\G$ with the Lebesgue measure in the 
$(a,b)$-coordinates $\lambda(da\, db)=da\, db$. 
We introduce coordinates $(u,s)\in\R^2$ of $\G$ by 
\begin{align}
a=e^{-u}\quad \text{ and }\quad b=s.
\end{align}
In these coordinates the Lebesgue measure $da\, db$ takes the form 
of the measure $\zeta$ from formula \eqref{eq:def-zeta}:
\begin{align}
da\, db=\zeta(du\, ds).
\end{align}

\paragraph{Right operation on $\G$.}
Note that this measure $\lambda$ is \emph{not} a Haar measure on $\G$. 
We define the right
operations 
\begin{align}
\cR_{v'}:\G\to\G,v\mapsto v''=v\cdot v'&\quad \text{ for } v'\in\G.
\end{align}
Under $\cR_{v'}$, using the notation $v''=[a'',b'']=[e^{-u''},s'']$, the measure $\lambda$ scales as follows:
\begin{align}
\cR_{[a',b']}[\lambda](da''\, db'') = & \frac{1}{a'}\, da''\, db''
=\frac{1}{a'}\,\zeta(du''\, ds'').
\label{eq:transf-right-reference-measure}
\end{align}

\paragraph{Cartesian power of $\G$.}
We define the cartesian power of the group $\G$ with one component 
pinned to the neutral element: 
\begin{align}
\label{eq:def-G-V}
\G_V:= \{ [a,b]:=([a_i,b_i])_{i\in \tilde V}\in\G^{\tilde V} :\; 
[a_\delta,b_\delta]=[1,0]\}.
\end{align}
In particular, the group operation $\cdot:\G_V\times\G_V\to\G_V$ is 
understood componentwise.
The set $\G_V$ can be identified with the set $\Omega_V$, defined in 
\eqref{eq:def-Omega}, via the componentwise
coordinate change to $(u,s)$-coordinates
\begin{align}
\label{eq:def-epsilon}
\iota:\G_V\to\Omega_V,\quad 
[a,b]\mapsto (-\log a,b). 
\end{align}

\paragraph{$\s$-operation as right operation.}
Using the identification $\iota$, the $\s$-operation is defined 
as right operation with inverse elements $[a,b]\in\G_V$:
\begin{align}
\label{eq:S-in-terms-of-right-operation}
\s_{[a,b]}:\Omega_V\to\Omega_V,\quad\s_{[a,b]}=& \iota\circ\cR_{[a,b]^{-1}}\circ\iota^{-1}
=\iota\circ\cR_{[1/a,-b/a]}\circ\iota^{-1},
\end{align}
i.e. for $(u,s)\in\Omega_V$, using the specifications \eqref{eq:group-op} and \eqref{eq:inverse} of the group operations,
\begin{align}
\s_{[a,b]}(u,s)=\iota\left(\iota^{-1}(u,s)\cdot [1/a,-b/a]\right)
=&\left(u_i+\log a_i,s_i- e^{-u_i}\frac{b_i}{a_i}\right)_{i\in \tilde V}. 
\label{eq:def-scaling-s}
\end{align}
Note that $[a_\delta,b_\delta]=[1,0]$ implies $\s_{[a,b]}(u,s)\in\Omega_V$. 
The map $\s:\G_V\times\Omega_V \to \Omega_V$,  $\s([a,b],(u,s))=\s_{[a,b]}(u,s)$, 
is a group action. Indeed, for $v_1,v_2,v\in\G_V$ it holds
\begin{align}
\label{eq:s-group-action}
\s_{v_1}(\s_{v_2}(\iota(v)))
= \iota((v\cdot v_2^{-1}) \cdot v_1^{-1})
= \iota(v\cdot (v_1\cdot v_2)^{-1})
=\s_{v_1\cdot v_2}(\iota(v)).
\end{align}
Moreover,
for the neutral element $[1,0]\in\G_V$ the map $\s_{[1,0]}$ is the identity. 
Consequently,
$\s_{[a,b]}$ is invertible for $[a,b]\in\G_V$ with the following inverse: 
\begin{align}
\label{eq:def-S-r-gamma-inv}
\s_{[a,b]}^{-1}=\s_{[a,b]^{-1}}:\;
\Omega_V&\to\Omega_V, \nonumber\\ 
(u,s)&\mapsto (\tilde u,\tilde s)= 
\iota(\iota^{-1}(u,s)\cdot[a,b])
=(u-\log a,s+e^{-u}b).
\end{align}
Let 
\begin{align}
\W=\{W=(W_{ij})_{i,j\in\tilde V}\in[0,\infty)^{\tilde V\times\tilde V}: 
W_{ij}=W_{ji}, W_{ij}>0 \text{ if and only if }i\sim j\}
\end{align} 
denote the set of 
possible edge weights. We consider the following group action on $\W$:
\begin{align}
\label{eq:def-rescaled-weights}
\G_V\times\W\to\W, ([a,b],W)\mapsto 
W^a=(W^a_{ij}:=a_ia_jW_{ij})_{i,j\in\tilde V}.
\end{align}

\subsection{Results in finite volume}
\label{subsect:finite}
Denote by $x_V$ the restriction of a vector $x\in\R^{\tilde V}$ to $\R^V$. 
Let 
\begin{align}
\label{eq:def-E}
e^{-u}_{VV}=\diag(e^{-u_i},i\in V)
\end{align}
denote the diagonal matrix in $\R^{V\times V}$ with entries $e^{-u_i}$ 
on the diagonal. 
We consider the variables 
$\theta^{V,W}(u,s)=(\theta^{V,W}_i(u,s))_{i\in V}$ defined by 
\begin{align}
\label{eq:def-theta} 
\theta^{V,W}(u,s)=e^{-u}_{VV}A^W_{VV}(u)s_V. 
\end{align}
Componentwise, we have for $i\in V$
\begin{align}
\label{eq:theta-componentwise}
\theta_i^{V,W}(u,s)
=\sum_{j\in\tilde V}W_{ij}e^{u_j}(s_i-s_j).
\end{align}
We need the random variables 
$\tilde\beta^{\tilde V,W}=(\tilde\beta_i^{\tilde V,W})_{i\in\tilde V}$ and
their restriction $\beta^{V,W}$ to $V$ defined by 
\begin{align}\label{eq:beta-definition}
\tilde\beta_i^{\tilde V,W}(u) =\frac12 \sum_{j\in \tilde V} W_{ij} e^{u_j-u_i}, 
\qquad \beta^{V,W}=\tilde\beta^{\tilde V,W}_V
=(\tilde\beta^{\tilde V,W}_i)_{i\in V}. 
\end{align}
These variables were introduced in \cite{sabot-tarres-zeng15}. 
We drop the dependence on $V$, $W$, or both if there is no risk of confusion. 

The following theorem describes the behavior of the supersymmetric
sigma model $\mu^W$ with respect to the scaling transformation 
$\s_{[a,b]}$ and is a fundamental ingredient
in this paper. Its extension to Grassmann variables is given 
in Theorem~\ref{thm:scaling-grassmann}.

\begin{theorem}
\label{thm:scaling}
Let $[a,b]\in\G_V$. The image of $\mu^{W^a}$ under the map $\s_{[a,b]}$ is 
absolutely 
continuous with respect to $\mu^W$ with the following Radon-Nikodym 
derivative on $\Omega_V$: 
\begin{align}
\frac{ d(\s_{[a,b]} \mu^{W^a}) }{d\mu^W} (u,s)
= & \cL^W(a,b)^{-1}e^{-\langle(a^2+b^2-1)_V,\beta^W(u)\rangle-\langle b_V,\theta^W(u,s)\rangle}
\label{eq:claim-thm-image-measure}
\end{align}
with the constant 
\begin{align}
\cL^W(a,b)
:= & \prod_{(i\sim j)\in \tilde E} e^{-W_{ij}(a_ia_j + b_ib_j -1)}
\cdot \prod_{j\in V} \frac{1}{a_j} .
\label{eq:claim-laplace-trafo}
\end{align}
In other words, for any measurable function $f:\Omega_V\to\R_0^+$, one has 
\begin{align}
\label{eq:claim-thm-image-measure-integral-version-new}
\E_{\mu^W}\left[ f(u,s) 
e^{-\langle(a^2+b^2-1)_V,\beta^W(u)\rangle-\langle b_V,\theta^W(u,s)\rangle} \right] 
=& \cL^W(a,b)
\E_{\mu^{W^a}}[ f \circ \s_{[a,b]}].
\end{align}
In particular, $\cL^W$ describes the joint Laplace transform of $\beta^W$ and 
$\theta^W$:
\begin{align}
\label{eq:L-a-b}
\cL^W(a,b)
= & \E_{\mu^W}\left[ 
e^{-\langle(a^2+b^2-1)_V,\beta^W(u)\rangle-\langle b_V,\theta^W(u,s)\rangle} \right] .
\end{align}
\end{theorem}
The special case $b=0$ was proven as Theorem 3.1 in 
\cite{disertori-merkl-rolles2017}. For $a=\sqrt{1+\lambda}$ and $b=0$
the Laplace transform $\cL^W(a,b)$ in \eqref{eq:L-a-b} equals 
the Laplace transform $\cL^W(\lambda)$ given by formula (2.9) in 
\cite{disertori-merkl-rolles2017}. 

\subsection{Results in infinite volume}
\label{subsect:infinite}
Let $G_\infty=(V_\infty,E_\infty)$ be an infinite graph with edge weights 
$W_{ij}$. We approximate $G_\infty$ by finite graphs with 
wired boundary conditions $\tilde G_n=(\tilde V_n,\tilde E_n)$, where 
$\tilde V_n=V_n\cup\{\delta_n\}$, $V_n\uparrow V_\infty$, and 
\begin{align}
\tilde E_n= E_n\cup\{(i\sim\delta_n):i\in V_n\text{ and }\exists 
j\in V_\infty\setminus V_n\text{ such that }(i\sim j)\in E_\infty\}. 
\end{align}
We endow the edges
of $\tilde G_n$ with the weights 
\begin{align}
\label{eq:def-weights-W-n-vertices}
& W^{(n)}_{ij}= W_{ij} \quad\text{ if } i\in V_n\text{ and }j\in V_n, \\
& W^{(n)}_{i\delta_n}=W^{(n)}_{\delta_ni}
=\sum_{j\in V_\infty\setminus V_n} W_{ij} 
\quad\text{ for }i\in V_n, \quad\text{and}\quad W^{(n)}_{\delta_n\delta_n}=0.
\label{eq:def-weights-W-n-delta}
\end{align}
Let $\mu_n^W$ denote the $H^{2|2}$ measure defined in \eqref{eq:def-mu-W} 
for the graph $\tilde G_n$ with the weights $W_{ij}^{(n)}$.

\begin{lemma}[Kolmogorov consistency]
\label{le:consistency}
For $n\in\N$, the joint Laplace transform 
\begin{align}
\cL^W_n(a,b)=
\E_{\mu^W_n}\left[e^{-\sk{(a^2+b^2-1)_{V_n},\beta^{V_n}}-\sk{b_{V_n},\theta^{V_n}}}\right]
\end{align}
of $\beta^{V_n}=(\beta_i)_{i\in V_n}$ and $\theta^{V_n}=(\theta_i)_{i\in V_n}$ satisfies 
the consistency relation 
\begin{align}
\label{eq:Ln-Ln+1equal}
\cL^W_n(a_{V_n},b_{V_n})=\cL^W_{n+1}(a,b),
\end{align}
for all $[a,b]\in\G_{V_{n+1}}$ with $[a_i,b_i]=[1,0]$ for 
all $i\in \tilde V_{n+1}\setminus V_n$. 
In particular, the law of $(\beta^{V_n},\theta^{V_n})$ with respect to $\mu^W_n$ agrees with 
the law of $(\beta^{V_{n+1}},\theta^{V_{n+1}})|_{V_n}$ with respect to $\mu^W_{n+1}$.
\end{lemma}
Consistency of the law of $\beta$ was first observed by Sabot and 
Zeng in \cite{sabot-zeng15}; see also Lemma 2.4 in 
\cite{disertori-merkl-rolles2017}.

By Kolmogorov's consistency theorem, there is a probability space 
$(\Omega_\infty,\F_\infty,\mu_\infty^W)$ with random variables 
$\pmb{\beta}_i,\pmb{\theta}_i:\Omega_\infty\to\R$, $i\in V_\infty$, such that for 
all $n\in\N$ the law of 
\begin{align}
\label{eq:def-beta-n-theta-n}
\big({\pmb{\beta}}^{(n)}=({\pmb{\beta}}_i)_{i\in V_n}, 
{\pmb{\theta}}^{(n)}=({\pmb{\theta}}_i)_{i\in V_n}\big)
\end{align}
with respect to $\mu_\infty^W$ agrees with the law of 
$(\beta^{V_n},\theta^{V_n}):\Omega_{V_n}\to\R^{V_n}\times\R^{V_n}$ 
with respect to $\mu_n^W$. 
Moreover, by Lemma 2.3 in \cite{disertori-merkl-rolles2017}, for any finite
graph $\tilde G=(\tilde V,\tilde E)$ with $\tilde V=V\cup\{\delta\}$, there is a 
measurable function $f_V^W:\R^V\to\R^{\tilde V}$ such that 
\begin{align}
\label{eq:u-as-fn-of-beta}
(u_i)_{i\in\tilde V}=f_V^W(\beta^V). 
\end{align}
Using the definition \eqref{eq:def-theta} of $\theta^W$, we have 
$s_V=A^W_{VV}(u)^{-1}e^u_{VV}\theta^V(u,s)$. Hence, 
\begin{align}
\label{eq:s-as-fn-of-beta-theta}
(s_i)_{i\in\tilde V} = g_V^W(\beta^V,\theta^V)
\end{align}
with the measurable function $g_V^W:\R^V\times\R^V\to\R^{\tilde V}$,
$(\beta,\theta)\mapsto s=(s_i)_{i\in\tilde V}$ defined by 
$s_\delta=0$ and $s_V = A^W_{VV}(f_V^W(\beta))^{-1}e^{f_V^W(\beta)}_{VV}\theta$.
This allows us to couple the $u$ and $s$-variables. We define 
\begin{align}
\label{eq:def-u-n}
& u^{(n)}=(u_i^{(n)})_{i\in\tilde V_n}=f_{V_n}^W({\pmb{\beta}}^{(n)}), \\
\label{eq:def-s-n}
& s^{(n)}=(s_i^{(n)})_{i\in\tilde V_n}=g_{V_n}^W({\pmb{\beta}}^{(n)},{\pmb{\theta}}^{(n)}), \\
& u_i^{(n)}=s_i^{(n)}=0
\quad\text{ for }i\in V_\infty\setminus V_n.
\label{eq:extend-u-n}
\end{align}

We consider the following set of parameters 
\begin{align}
\label{eq:def-parameter-space}
(-\infty,0]^{(V_\infty)}
=\{ \alpha\in(-\infty,0]^{V_\infty}: \alpha_i\neq 0\text{ for only finitely many }i\in V_\infty
\}.
\end{align}
For $\alpha\in (-\infty,0]^{(V_\infty)}$ and $n\in\N$, we define 
$\alpha^{(n)}=(\alpha^{(n)}_i)_{i\in\tilde V_n}$ by 
\begin{align}
\label{eq:def-a-n}
\alpha^{(n)}_i=\alpha_i\quad\text{for }i\in V_n \qquad\text{and}\qquad
\alpha^{(n)}_{\delta_n}=\sum_{j\in V_\infty\setminus V_n}\alpha_j.
\end{align}

\begin{theorem}
\label{thm:generating-mg}
For all $\alpha\in (-\infty,0]^{(V_\infty)}$, the sequence $(M^{(n)}_\alpha)_{n\in\N}$,
defined by 
\begin{align}
\label{eq:def-M-n-theta}
M^{(n)}_\alpha:(u^{(n)},s^{(n)})\mapsto 
\exp\left( \sum_{j\in\tilde V_n} \alpha^{(n)}_je^{u^{(n)}_j}(1+is^{(n)}_j)
\right),
\end{align}
is a $\C$-valued martingale with respect to the filtration 
$\big(\F_n=\sigma(\pmb{\beta}^{(n)},{\pmb{\theta}}^{(n)})\big)_{n\in\N}$. 
\end{theorem}

Taking derivatives of the martingale $(M^{(n)}_\alpha)_{n\in\N}$ at
$\alpha=0$, we obtain the following hierarchy of martingales.

\begin{corollary}
\label{eq:cor-derivatives-are-mgs}
For all $k\in\N$ and $j_1,\ldots,j_k\in V_\infty$, 
\begin{align}
M_{j_1,\ldots,j_k}^{(n)}=\prod_{l=1}^k e^{u_{j_l}^{(n)}}(1+is_{j_l}^{(n)}), 
\quad
n\in\N, 
\end{align}
its real and imaginary part are martingales with respect to 
$\big(\F_n=\sigma(\pmb{\beta}^{(n)},{\pmb{\theta}}^{(n)})\big)_{n\in\N}$. 
\end{corollary}

\noindent In \cite{disertori-merkl-rolles2017}, we showed that 
$(\E_{\mu^W_\infty}[M^{(n)}_\alpha|\sigma(u^{(n)})])_{n\in\N}$ and 
$(\E_{\mu^W_\infty}[M_{j_1,\ldots,j_k}^{(n)}|\sigma(u^{(n)})])_{n\in\N}$ are martingales. These facts are also
immediate consequences of Theorem~\ref{thm:generating-mg} 
and Corollary~\ref{eq:cor-derivatives-are-mgs}.

\section{The marginal $\mu^W (du\, ds)$}
\label{sect:2}

We start with two different proofs for Theorem~\ref{thm:scaling}.
\begin{itemize}
\item The first proof is
based on Lemma~\ref{le:ratio-of-densities} below which describes the ratio
between the original and $\s$-transformed probability density of two supersymmetric
sigma models with different parameters. 
Also for this lemma two different proofs
are given. 
\begin{itemize}
\item The first proof is based on explicit computations on the quadratic form
associated to the matrix $A^{W}$ defined in equation \eqref{eq:def-A}.  
\item The second proof uses the description of the density of the supersymmetric
sigma model in terms of $2\times 2$ determinants connected to the linear algebra
of Weyl spinors.
\end{itemize}
Both these proofs are self-contained. 
\item The second proof of Theorem~\ref{thm:scaling}  uses conditioning on the $u$
variables and a result from \cite{disertori-merkl-rolles2017}.
\end{itemize}
Finally Section \ref{subsec:martin} contains the proof of the martingale property.

\subsection{First proof of Theorem~\ref{thm:scaling}}

The group $\G_V$ introduced in formula \eqref{eq:def-G-V} acts  on the set $\M$ of 
measures on $\Omega_V$:
\begin{align}
\G_V\times\M\to\M,\quad (v,\mu)\mapsto \s_v\mu,
\end{align}
where $\s_v\mu$ denotes the image measure of $\mu$ with respect to 
the map $\s_v:\Omega_V\to\Omega_V$. 
In particular, using the measure $\zeta$ introduced in formula
\eqref{eq:def-zeta}, we consider the product 
\begin{align}
\label{eq:def-zeta-V}
\zeta_V:=\zeta^V\times\delta_{(0,0)}
\end{align}
composed of factors $\zeta$ indexed by $V$ and one Dirac measure located at 
$(0,0)\in\R^2$ indexed by the special vertex $\delta$. 

\begin{lemma}
\label{le:zeta-transformed-with-s}
For $[a,b]\in\G_V$, one has 
\begin{align}
\label{eq:zeta-transformed-with-s}
\s_{[a,b]}\zeta_V=\left(\prod_{i\in V}a_i\right)\zeta_V
\end{align}
\end{lemma}
\begin{proof}
This is an immediate consequence of formulas 
\eqref{eq:S-in-terms-of-right-operation}
and \eqref{eq:transf-right-reference-measure}. 
\end{proof}

\begin{lemma}[Ratio of densities]
\label{le:ratio-of-densities}
For $[a,b]\in\G_V$ and $(u,s)\in\Omega_V$, one has 
\begin{align}
\label{eq:ratio-of-densities-rho}
\frac{\rho^{W^a}(\s^{-1}_{[a,b]}(u,s))}{\rho^W(u,s)}
=\prod_{(i\sim j)\in \tilde E} e^{W_{ij}(a_ia_j + b_ib_j -1)}
\prod_{i\in V}\exp[-(a_i^2+b_i^2-1)\beta_i^W(u)-b_i\theta_i^W(u,s) ]. 
\end{align}
\end{lemma}

This lemma is proven in Section \ref{se:two-proofs-ratio-densities}, below.

\medskip\noindent\begin{proof}[First proof of Theorem~\ref{thm:scaling}]
We abbreviate $c=(2\pi)^{-|V|}$. From \eqref{eq:def-mu-W}, we know 
$d\mu^W=c\rho^W\, d\zeta_V$. Substituting $W$ by $W^a$, this gives 
$d\mu^{W^a}=c\rho^{W^a}\, d\zeta_V$. We take now the image measure with respect
to $\s_{[a,b]}$. The following calculation uses the description 
of $\s_{[a,b]}\zeta_V$ from Lemma~\ref{le:zeta-transformed-with-s} 
and in the last step the ratio of densities given in Lemma 
\ref{le:ratio-of-densities} together with the definition 
\eqref{eq:claim-laplace-trafo} of the constant $\cL^W(a,b)$.
\begin{align}
& d(\s_{[a,b]}\mu^{W^a})
= 
c (\rho^{W^a}\circ\s^{-1}_{[a,b]})\, d(\s_{[a,b]}\zeta_V)
=
c\frac{\rho^{W^a}\circ\s^{-1}_{[a,b]}}{\rho^W}
\rho^W\prod_{i\in V}a_i\,d\zeta_V \nonumber\\
= &\frac{\rho^{W^a}\circ\s^{-1}_{[a,b]}}{\rho^W}\prod_{i\in V}a_i\, d\mu^W
= \cL^W(a,b)^{-1}e^{-\langle(a^2+b^2-1)_V,\beta^W\rangle-\langle b_V,\theta^W\rangle}
\, d\mu^W
\end{align}
This implies the claim \eqref{eq:claim-thm-image-measure}, which is 
written in \eqref{eq:claim-thm-image-measure-integral-version-new} in a 
different notation. Taking the test function $f=1$, 
\eqref{eq:L-a-b} is a special case of 
\eqref{eq:claim-thm-image-measure-integral-version-new}. 
\end{proof}

\subsection{Two proofs of Lemma~\ref{le:ratio-of-densities}}
\label{se:two-proofs-ratio-densities}

The two proofs of Lemma~\ref{le:ratio-of-densities} given in this 
subsection are based on 
two different representations of $\rho^W$. The first representation,
being based on the quadratic form associated to $A^W$,
is given in Lemma~\ref{le:expr1-rho}, while the second one, being
based on $2\times 2$ determinants, is described in Lemma 
\ref{le:rho-in-terms-of-det}, below. 

\subsubsection{First proof}
We define the matrix $H_{\tilde\beta(u)}^W\in\R^{\tilde V\times\tilde V}$ by 
\begin{align}
& (H^W_{\tilde\beta(u)})_{ij}=2\tilde\beta_i(u)\delta_{ij}-W_{ij}\quad 
\text{ for }i,j\in\tilde V,
\end{align}
Note that for all $i,j\in\tilde V$, one has 
\begin{align}
(H^W_{\tilde\beta(u)})_{ij}=& \left\{\begin{array}{ll} 
-W_{ij} & \text{ if }i\neq j, \\
2\tilde\beta_i(u)=\sum_{k\in\tilde V}W_{ik}e^{u_k-u_i}& \text{ if }i=j
\end{array}\right\}\nonumber\\
= & e^{-u_i-u_j}A_{ij}^W(u)=
(e^{-u}A^W(u)e^{-u})_{ij};
\label{eq:def-H-beta}
\end{align}
recall that the graph $\tilde G$ has no direct loops and hence 
$W_{ii}=0$ by the definition of the weights. Here and in the 
following, when calculating with matrices, we abbreviate 
$e^{\pm u}=\diag(e^{\pm u_i},i\in\tilde V)$. Thus, expressions like 
$e^{-u}s$ can be read in two equivalent ways, componentwise or 
as a matrix multiplication, both meaning the same object
$(e^{-u_i}s_i)_{i\in\tilde V}$. In contrast to this, we write 
$e^{-u}_{\tilde V}=(e^{-u_i})_{i\in\tilde V}$ for the column vector. 
We denote by $H^W_{\beta(u)}:=(H^W_{\tilde\beta(u)})_{VV}$ the restriction 
to $V\times V$, i.e.\ $(H^W_{\beta(u)})_{ij}=2\beta_i(u)\delta_{ij}-W_{ij}$ for 
$i,j\in V$, cf.\ \eqref{eq:beta-definition}.

\begin{lemma}
\label{le:expr1-rho}
For $(u,s)\in\Omega_V$, we have the relations 
\begin{align}
&2\sum_{(i\sim j)\in\tilde E}W_{ij}[\cosh(u_i-u_j)-1]
=\sk{e^{-u}_{\tilde V},A^{W}(u)e^{-u}_{\tilde V}}=\sk{1_{\tilde V},H^W_{\tilde\beta(u)} 1_{\tilde V}}, 
\label{eq:expr2-rho}\\
&\det A^W_{VV}(u)=\prod_{i\in V}e^{2u_i}\cdot \det H_{\beta(u)}^W,\quad 
\sk{s,A^{W}(u)s}=\sk{e^us,H^W_{\tilde\beta(u)}e^us}.
\label{eq:expr3-rho}
\end{align}
In particular, the density $\rho^W$ defined in \eqref{eq:def-rho} 
can be written as
\begin{align}
\label{eq:expr1-rho}
\rho^W(u,s)=& \det A^W_{VV}(u) 
e^{-\frac12 \sk{s,A^{W}(u)s}} e^{-\frac12 \sk{e^{-u}_{\tilde V},A^{W}(u)e^{-u}_{\tilde V}}}.
\end{align}
\end{lemma}
\begin{proof}
The first equation in \eqref{eq:expr2-rho} follows directly from 
\begin{align}
\sk{e^{-u}_{\tilde V},A^W(u)e^{-u}_{\tilde V}}
=&\sum_{i,j\in\tilde V}e^{-u_i}A^W_{ij}(u)e^{-u_j}
=\sum_{i\in\tilde V}\sum_{k\in\tilde V}W_{ik}e^{u_k-u_i}
-2\sum_{(i\sim j)\in\tilde E}W_{ij}\label{eq:expr-gauss}\\
=&2\sum_{(i\sim j)\in\tilde E}W_{ij}[\cosh(u_i-u_j)-1],\nonumber
\end{align}
where the first sum on the right-hand side of \eqref{eq:expr-gauss} 
comes from the diagonal terms 
in $A^W(u)$ and the second sum from the off-diagonal terms. 
Using relation \eqref{eq:def-H-beta} between $H^W_{\tilde\beta(u)}$ 
and $A^W(u)$, the remaining claims in \eqref{eq:expr2-rho} and
\eqref{eq:expr3-rho} follow. The expression \eqref{eq:expr1-rho}
for $\rho^W$ is a consequence of formula \eqref{eq:expr2-rho}. 
\end{proof}

\begin{lemma}
\label{le:transformed-A}
The matrix $A^W$ is invariant with respect to the 
$\s$-operation in the following sense: For $[a,b]\in\G_V$,
$(u,s)\in\Omega_V$, and 
$(\tilde u,\tilde s)=\s^{-1}_{[a,b]}(u,s)=(u-\log a,s+e^{-u}b)$, 
the following holds 
\begin{align}
\label{eq:transformed-A}
A^{W^a}(\tilde u)=A^W(u), \quad\text{i.e.}\quad
A^W = A^{W^a} \circ \s^{-1}_{[a,b]}.
\end{align}
\end{lemma}
\begin{proof}
For $i,j\in\tilde V$ with $i\neq j$, one has 
$A^{W^a}_{ij}(\tilde u)=a_ia_jW_{ij}e^{\tilde u_i+\tilde u_j}=
W_{ij}e^{u_i+u_j}=A^W_{ij}(u)$. Since rows of both matrices $A^{W^a}(\tilde u)$
and $A^W(u)$ sum up to 0, it follows also $A^{W^a}_{ii}(\tilde u)=A^W_{ii}(u)$.
This proves the claim. 
\end{proof}

\medskip\noindent
\begin{proof}[First proof of Lemma~\ref{le:ratio-of-densities}]
Substituting \eqref{eq:transformed-A} into the representation 
\eqref{eq:expr1-rho} for $\rho^{W^a}$, we obtain 
\begin{align}
\rho^{W^a}(\s^{-1}_{[a,b]}(u,s))=\rho^{W^a}(\tilde u,\tilde s)
=& \det A^W_{VV}(u) e^{-\frac12 \sk{\tilde s,A^{W}(u)\tilde s}} 
e^{-\frac12 \sk{e^{-\tilde u}_{\tilde V},A^{W}(u)e^{-\tilde u}_{\tilde V}}}.
\end{align}
Inserting the definition of $\tilde u$ and $\tilde s$ in the exponents 
above and using \eqref{eq:def-H-beta}, the facts $b_\delta=0=s_\delta$ and 
the definition \eqref{eq:def-theta} of $\theta^W$, we obtain 
\begin{align}
\sk{\tilde s,A^{W}(u)\tilde s}
=&\sk{s,A^{W}(u)s} + \sk{b,e^{-u}A^{W}(u)e^{-u}b} 
+2\sk{b,e^{-u}A^{W}(u)s}
\nonumber\\
\label{eq:trafo-gauss}
=&\sk{s,A^{W}(u)s} + \sk{b,H_{\tilde\beta(u)}^Wb} +2\sk{b_V,\theta^W(u,s)}, \\
\label{eq:trafo-cosh}
\sk{e^{-\tilde u}_{\tilde V},A^{W}(u)e^{-\tilde u}_{\tilde V}}
= & \sk{a,e^{-u}A^{W}(u)e^{-u}a}=\sk{a,H_{\tilde\beta(u)}^Wa}.
\end{align}
Using in the second equality \eqref{eq:expr1-rho} and 
\eqref{eq:expr2-rho}, this implies 
\begin{align}
\rho^{W^a}(\s^{-1}_{[a,b]}(u,s))
=& \det A^W_{VV}(u) 
e^{-\frac12 \left(\sk{s,A^{W}(u)s} + \sk{b,H_{\tilde\beta(u)}^Wb}\right)
-\sk{b_V,\theta^W(u,s)}} 
e^{-\frac12\sk{a,H_{\tilde\beta(u)}^Wa}} \nonumber\\
= & \rho^W(u,s)
e^{-\frac12 \left(\sk{a,H_{\tilde\beta(u)}^Wa}
+\sk{b,H_{\tilde\beta(u)}^Wb}
-\sk{1_{\tilde V},H^W_{\tilde\beta(u)} 1_{\tilde V}}\right)}e^{-\sk{b_V,\theta^W(u,s)}}.
\end{align}
Since $a_\delta^2+b_\delta^2-1=0$, the first exponent in the 
last expression takes the form 
\begin{align}
& -\frac12\left(\sk{a,H_{\tilde\beta(u)}^Wa}
+\sk{b,H_{\tilde\beta(u)}^Wb}
-\sk{1_{\tilde V},H^W_{\tilde\beta(u)} 1_{\tilde V}}\right) \nonumber\\
= & \sum_{(i\sim j)\in \tilde E} W_{ij}(a_ia_j + b_ib_j -1)
-\sum_{i\in V}(a_i^2+b_i^2-1)\beta_i^W(u).
\end{align}
This proves the claim. 
\end{proof}

\subsubsection{Second proof}
We can represent the density $\rho^W$ of 
the supersymmetric sigma model as follows. 
Recall the bijection $\iota$ introduced in \eqref{eq:def-epsilon}. 

\begin{lemma}
\label{le:rho-in-terms-of-det}
For $(u,s)=\iota(v)\in\Omega_V$ with $v=[a,b]\in\G_V$, the density $\rho^W$ 
defined in \eqref{eq:def-rho} can be written as follows: 
\begin{align}
\rho^W(u,s)=\det A^W_{VV}(u) 
\exp\left(\sum_{(i\sim j)\in \tilde E} \frac{W_{ij}}{2}
\det\left(\frac{v_iv_i^t}{a_i}-\frac{v_jv_j^t}{a_j}\right)\right)
\end{align}
\end{lemma}
\begin{proof}
Let $(u,s)=\iota(v)\in\Omega_V$. It suffices to prove for all 
$(i\sim j)\in \tilde E$ 
\begin{align}
\label{eq:claim-edge-wise}
-\left[\cosh(u_i-u_j)-1+\frac12(s_i-s_j)^2e^{u_i+u_j}\right]
=\frac12\det\left(\frac{v_iv_i^t}{a_i}-\frac{v_jv_j^t}{a_j}\right). 
\end{align}
For $i\in\tilde V$, $v_i=[e^{-u_i},s_i]=[a_i,b_i]$, we calculate 
\begin{align} 
\frac{v_iv_i^t}{a_i}
=e^{u_i}\matr{{cc} e^{-u_i} & s_i \\ 0 & 1}\matr{{cc} e^{-u_i} & 0\\ s_i & 1}
=\matr{{cc} e^{-u_i} +s_i^2e^{u_i} & s_ie^{u_i}\\ s_ie^{u_i} & e^{u_i}}.
\end{align}
Consequently, the claim \eqref{eq:claim-edge-wise} follows from 
\begin{align} 
\det\left(\frac{v_iv_i^t}{a_i}-\frac{v_jv_j^t}{a_j}\right)
=& (e^{-u_i}-e^{-u_j}+s_i^2e^{u_i}-s_j^2e^{u_j})(e^{u_i}-e^{u_j})
-(s_ie^{u_i} -s_je^{u_j})^2\nonumber\\
=& 2-2\cosh(u_i-u_j)-(s_i-s_j)^2e^{u_i+u_j}. 
\end{align}
\end{proof}

To deal with determinants of differences of $2\times 2$-matrices, 
we need the following elementary lemma, which is motivated 
by the linear algebra of spinors.  Let 
\begin{align}
\eps=\matr{{cc}0 & -1 \\ 1 & 0}. 
\end{align}

\begin{lemma}
\label{le:diff-det-2}
For all $v_i=[a_i,b_i],v_j=[a_j,b_j]\in\G$, one has 
\begin{align}
\label{eq:diff-det-2}
\det\left(\frac{v_iv_i^t}{a_i}-\frac{v_jv_j^t}{a_j}\right)
= 2-\frac{\|v_i^t\eps v_j\|^2}{a_ia_j}, 
\end{align}
where $\|\cdot\|$ means the euclidean norm of $2\times 2$-matrices.
\end{lemma}
\begin{proof}
The bilinear form 
$\tr(A\eps B^t\eps)$ on $2\times 2$-matrices $A,B\in\R^{2\times 2}$
is symmetric. Indeed, using $\eps^t=-\eps$, 
\begin{align}
\tr(B\eps A^t\eps)=\tr((B\eps A^t\eps)^t)=\tr(\eps A\eps B^t)
=\tr(A\eps B^t\eps). 
\end{align}
The corresponding quadratic form is given by 
\begin{align}
\tr(A\eps A^t\eps)=-2\det A. 
\end{align}
It follows 
\begin{align}
\det(A-B)=\det A+\det B+\tr(A\eps B^t\eps). 
\end{align}
Taking now $A=a_i^{-1}v_iv_i^t$ and $B=a_j^{-1}v_jv_j^t=B^t$, which fulfill
$\det A=a_i^{-2}(\det v_i)^2=1=\det B$, we obtain 
\begin{align}
\det\left(\frac{v_iv_i^t}{a_i}-\frac{v_jv_j^t}{a_j}\right)
=&\det (A-B)=\det A+\det B+\tr(A\eps B^t\eps)\nonumber\\
= & 2+\tr(A\eps B\eps)
= 2+\frac{1}{a_ia_j}\tr(v_iv_i^t\eps v_jv_j^t\eps).
\label{eq:det-diff}
\end{align}
Using $\eps^t=-\eps$ again, we rewrite the last trace as follows: 
\begin{align}
\tr(v_iv_i^t\eps v_jv_j^t\eps)
=\tr(v_i^t\eps v_jv_j^t\eps v_i)
=-\tr(v_i^t\eps v_j(v_i^t\eps v_j)^t)
=-\|v_i^t\eps v_j\|^2.
\end{align}
Substituting this into \eqref{eq:det-diff}, the claim 
\eqref{eq:diff-det-2} follows. 
\end{proof}

\medskip\noindent
\begin{proof}[Second proof of Lemma~\ref{le:ratio-of-densities}]
We take $v=[a,b]$, $v'=[a',b']$, and $v''=[a'',b'']$ in $\G_V$
with $v''=v'\cdot v$ and set  
$(u,s)=\iota(v')$, $(\tilde u,\tilde s)=\iota(v'')$. 
By \eqref{eq:def-S-r-gamma-inv}, we have $\s^{-1}_{[a,b]}(u,s)=\iota(v'')$. 
Since $A^{W^a}(\tilde u)=A^W(u)$ as stated in Lemma~\ref{le:transformed-A}, 
it follows 
\begin{align}
\det A^{W^a}_{VV}(\tilde u)=\det A^W_{VV}(u).
\label{eq:det-equal}
\end{align} 
Using Lemma~\ref{le:rho-in-terms-of-det} and this fact, 
we obtain
\begin{align}
& \frac{\rho^{W^a}(\s^{-1}_{[a,b]}(u,s))}{\rho^W(u,s)}
= \frac{\rho^{W^a}(\iota([a'',b''])}{\rho^W(\iota([a',b'])}
\nonumber\\
& =\exp\left(\sum_{(i\sim j)\in \tilde E} \frac{W_{ij}}{2}
\left[a_ia_j\det\left(\frac{v_i''(v_i'')^t}{a_i''}
-\frac{v_j''(v_j'')^t}{a_j''}\right)
-\det\left(\frac{v_i'(v_i')^t}{a_i'}
-\frac{v_j'(v_j')^t}{a_j'}\right)\right]\right)
\label{eq:ratio-densities1}
\end{align}
We apply Lemma~\ref{le:diff-det-2} to $v'=[a',b']$ and $v''=[a'',b'']$
as follows, using $a_i''=a_ia_i'$ for $i\in\tilde V$:
\begin{align}
\label{eq:diff-det-3}
& a_ia_j\det\left(\frac{v_i''(v_i'')^t}{a_i''}
-\frac{v_j''(v_j'')^t}{a_j''}\right)
-\det\left(\frac{v_i'(v_i')^t}{a_i'}
-\frac{v_j'(v_j')^t}{a_j'}\right)\nonumber\\
=& 2a_ia_j-\frac{a_ia_j}{a_i''a_j''}\|(v_i'')^t\eps v_j''\|^2
-2+\frac{\|(v_i')^t\eps v_j'\|^2}{a_i'a_j'}\nonumber\\
=& 2(a_ia_j-1)+\frac{1}{a_i'a_j'}\left( \|(v_i')^t\eps v_j'\|^2
-\|(v_i'')^t\eps v_j''\|^2\right). 
\end{align}
Note that 
\begin{align}
& (v_i')^t\eps v_j'=\matr{{cc} 0 & - a_i' \\
a_j' & b_j'-b_i'}\quad\text{and}\quad \\
& (v_i'')^t\eps v_j''=\matr{{cc} 0 & - a_i'' \\
a_j'' & b_j''-b_i''}
=\matr{{cc} 0 & - a_ia_i' \\
a_ja_j' & b_j'-b_i'+a_j'b_j-a_i'b_i}
\end{align}
We calculate the last parenthesis in \eqref{eq:diff-det-3}, writing $\sk{\cdot,\cdot}$ for the euclidean
scalar product of matrices:
\begin{align}
& \|(v_i')^t\eps v_j'\|^2-\|(v_i'')^t\eps v_j''\|^2
=\sk{(v_i')^t\eps v_j'+(v_i'')^t\eps v_j'',(v_i')^t\eps v_j'-(v_i'')^t\eps v_j''}
\nonumber\\
& =\sk{
\begin{pmatrix}0 & -a_i'(1+a_i) \\
a_j'(1+a_j) & -2(b_i'-b_j')-(a_i'b_i-a_j'b_j)
\end{pmatrix},
\begin{pmatrix}0 & -a_i'(1-a_i) \\
a_j'(1-a_j) & a_i'b_i-a_j'b_j
\end{pmatrix}} 
\nonumber\\
& = -a_i'a_j'\left(
\frac{a_i'}{a_j'}(a_i^2-1)+\frac{a_j'}{a_i'}(a_j^2-1)
+2(b_i'-b_j')\left(\frac{b_i}{a_j'}-\frac{b_j}{a_i'}\right)
+a_i'a_j'\left(\frac{b_i}{a_j'}-\frac{b_j}{a_i'}\right)^2
\right). 
\end{align}
This yields 
\begin{align}
\text{l.h.s.\ in }\eqref{eq:diff-det-3}
= & 2(a_ia_j+b_ib_j-1)
-\left(\frac{a_i'}{a_j'}(a_i^2+b_i^2-1)
+\frac{a_j'}{a_i'}(a_j^2+b_j^2-1)\right)\nonumber\\
& -2\frac{b_i}{a_j'}(b_i'-b_j') -2\frac{b_j}{a_i'}(b_j'-b_i')
. 
\end{align}
Multiplying this with $W_{ij}/2$, summing the result over 
$(i\sim j)\in\tilde E$, and using the symmetry $W_{ij}=W_{ji}$, 
we obtain 
\begin{align}
& \sum_{(i\sim j)\in\tilde E}\frac{W_{ij}}{2}\left[
a_ia_j\det\left(\frac{v_i''(v_i'')^t}{a_i''}
-\frac{v_j''(v_j'')^t}{a_j''}\right)
-\det\left(\frac{v_i'(v_i')^t}{a_i'}
-\frac{v_j'(v_j')^t}{a_j'}\right)\right]\nonumber\\
= & \sum_{(i\sim j)\in\tilde E}W_{ij}\left[
a_ia_j+b_ib_j-1\right]
-\sum_{i,j\in \tilde V}W_{ij}\left[
\frac12\frac{a_i'}{a_j'}(a_i^2+b_i^2-1)+
\frac{b_i}{a_j'}(b_i'-b_j')
\right]
\label{eq:diff-det-4}
\end{align}
Next, we rewrite the definitions \eqref{eq:beta-definition} 
of $\beta^W_i$ and \eqref{eq:theta-componentwise} of $\theta^W_i$, 
$i\in V$, in the following form, using $[a',b']=[e^{-u},s]$: 
\begin{align}
& \beta_i^W =\frac12 \sum_{j\in \tilde V} W_{ij} e^{u_j-u_i}
=\frac12 \sum_{j\in \tilde V} W_{ij} \frac{a_i'}{a_j'}, \\
& \theta^W_i=\sum_{j\in\tilde V}W_{ij}e^{u_j}(s_i-s_j)
=\sum_{j\in\tilde V}W_{ij}\frac{1}{a_j'}(b_i'-b_j'). 
\end{align}
Since $a_\delta^2+b_\delta^2-1=0$, $b_\delta=0$, we obtain 
\begin{align}
\text{l.h.s.\ in }\eqref{eq:diff-det-4}
= & \sum_{(i\sim j)\in\tilde E}W_{ij}\left[
a_ia_j+b_ib_j-1\right]
-\sum_{i\in V}\left[(a_i^2+b_i^2-1)\beta_i^W
+b_i\theta_i^W\right],
\end{align}
Substituting this into \eqref{eq:ratio-densities1}, the claim 
\eqref{eq:ratio-of-densities-rho} follows.
\end{proof}

\subsection{Proof of Theorem~\ref{thm:scaling} by conditioning}

The following lemma describes the conditional distribution of $\theta^W$ 
given $\beta^W$. 

\begin{lemma}
\label{le:distribution-theta-given-u}
Conditioned on $\beta^W$, the random vector $\theta^W\in\R^V$ is normally 
distributed with mean $0$ and covariance matrix 
\begin{align}
\label{eq:relation-H-beta-A}
H^W_{\beta(u)}=e^{-u}_{VV}A^W_{VV}(u)e^{-u}_{VV}. 
\end{align} 
\end{lemma}
\begin{proof}
By definition, conditioned on $u$, the vector $s_V$ is centered Gaussian with 
covariance matrix $A^{-1}$, where $A:=A_{VV}^W(u)$. Since $u$ is a 
function of $\beta^W$ by Lemma 2.3 of \cite{disertori-merkl-rolles2017}, 
we have conditioned on $\beta^W$ that $\theta^W=e^{-u}_{VV}As_V$ 
is also centered Gaussian with covariance matrix 
$(e^{-u}_{VV}A) A^{-1} (e^{-u}_{VV}A)^t=e^{-u}_{VV}Ae^{-u}_{VV}$. 
The representation \eqref{eq:relation-H-beta-A} follows from 
\eqref{eq:def-H-beta}.
\end{proof}

Our second proof of Theorem~\ref{thm:scaling} uses the known 
transformation behavior of $\mu^{W^a}(du\, ds)$ with respect to 
$\s_{[a,0]}$ from \cite{disertori-merkl-rolles2017} and the fact 
that conditionally on $u$, the $s$-variables are jointly Gaussian. 

\medskip\noindent
\begin{proof}[Proof of Theorem~\ref{thm:scaling} by conditioning]
To prove \eqref{eq:claim-thm-image-measure-integral-version-new}, it 
suffices to consider test functions of the form 
$f(u,s)=g(u)h(s)$ with measurable functions $g,h:\R^{\tilde V}\to\R^+_0$. 
We calculate
\begin{align}
\E_{\mu^{W^a}}\left[ f \circ \s_{[a,b]}\right]
=& \E_{\mu^{W^a}}\left[ g(u+\log a)h(s-e^{-(u+\log a)}b)\right]
\end{align}
The behavior of the supersymmetric sigma model $\mu^{W^a}$ with 
rescaled weights with respect to the shift $u\mapsto u+\log a$ in the $u$ 
variables was studied in \cite{disertori-merkl-rolles2017}. Using 
Theorem 3.1 of that paper with $\lambda=a^2-1$ yields 
\begin{align}
&\E_{\mu^{W^a}}\left[ g(u+\log a)h(s-e^{-(u+\log a)}b)\right]
=\cL^W(a,0)^{-1}\E_{\mu^W}\left[ g(u)h(s-e^{-u}b)e^{-\sk{(a^2-1)_V,\beta^W(u)}}\right]
\nonumber\\
& = \cL^W(a,0)^{-1}
\E_{\mu^W}\left[ g(u)\E_{\mu^W}[h(s-e^{-u}b)|u]e^{-\sk{(a^2-1)_V,\beta^W(u)}}\right]
\label{eq:large-expectation}
\end{align}
with the constant $\cL^W(a,0)$ given in \eqref{eq:claim-laplace-trafo}; 
recall that $\beta^W$ is a function of $u$. 
By the definition of the supersymmetric sigma model, cf.\  
\eqref{eq:def-mu-W} and \eqref{eq:def-rho}, conditioned on $u$ the vector
$s_V$ is centered Gaussian with covariance matrix $A^W_{VV}(u)^{-1}$
and $s_\delta=0$. Consequently, abbreviating 
$c=(2\pi)^{-|V|/2}\sqrt{\det A^W_{VV}(u)}$ and 
$\sigma(ds)=\delta_0(ds_\delta)\prod_{i\in V}ds_i$, 
the conditional expectation in \eqref{eq:large-expectation} is given by 
\begin{align}
& \E_{\mu^W}[h(s-e^{-u}b)|u]
=  c\int_{\R^{\tilde V}}h(s-e^{-u}b)
e^{-\frac12\sk{s,A^W(u)s}}\,\sigma(ds) \nonumber\\
= & c\int_{\R^{\tilde V}}h(s)
e^{-\frac12\sk{s+e^{-u}b,A^W(u)(s+e^{-u}b)}}\, 
\sigma(ds)\nonumber\\
= & c\int_{\R^{\tilde V}}h(s)
e^{-\frac12\sk{s,A^W(u)s}-\frac12\sk{e^{-u}b,A^W(u)e^{-u}b}-\sk{e^{-u}b,A^W(u)s}}\, \sigma(ds). 
\label{eq:cond-expect-of-h-given-u}
\end{align}
Using $b_\delta=0$ and \eqref{eq:def-H-beta}, we obtain 
\begin{align}
\label{eq:scalar-product1}
\sk{e^{-u}b,A^W(u)e^{-u}b}
=\sk{b_V,H^W_{\beta(u)}b_V}
=2\sk{b^2_V,\beta^W(u)}-\sum_{i,j\in V}W_{ij}b_ib_j. 
\end{align}
Similarly, using the definition \eqref{eq:def-theta}  of $\theta^W$, we obtain 
\begin{align}
\label{eq:scalar-product2}
\sk{e^{-u}b,A^W(u)s}
=\sk{b_V,e^{-u}_{VV}A^W_{VV}(u)s_V}
=\sk{b_V,\theta^W(u,s)}. 
\end{align}
Inserting \eqref{eq:scalar-product1} and \eqref{eq:scalar-product2} into 
\eqref{eq:cond-expect-of-h-given-u} yields
\begin{align}
\E_{\mu^W}[h(s-e^{-u}b)|u]
= & \prod_{i,j\in V}e^{\frac12W_{ij}b_ib_j} \cdot e^{-\sk{b^2_V,\beta^W(u)}}
c\int_{\R^{\tilde V}}h(s)
e^{-\frac12\sk{s,A^W(u)s}-\sk{b_V,\theta^W(u,s)}}\, \sigma(ds) \nonumber\\
= & \cL^W(a,0)\cL^W(a,b)^{-1} \cdot e^{-\sk{b^2_V,\beta^W(u)}}
\E_{\mu^W}\left[\left. h(s)e^{-\sk{b_V,\theta^W(u,s)}}
\right|u\right].
\end{align}
Inserting the above in \eqref{eq:large-expectation} yields the claim 
\eqref{eq:claim-thm-image-measure-integral-version-new}. 
Equality \eqref{eq:L-a-b} follows from 
\eqref{eq:claim-thm-image-measure-integral-version-new} applied
to the function $f(u,s)=1$. 
\end{proof}

\subsection{Martingales}
\label{subsec:martin}

\medskip\noindent
\begin{proof}[Proof of Kolmogorov consistency (Lemma~\ref{le:consistency})]
By Theorem~\ref{thm:scaling}, one has 
\begin{align}
\label{eq:L-n}
\cL^W_n(a_{V_n},b_{V_n})
= & \prod_{(i\sim j)\in \tilde E_n} e^{-W_{ij}^{(n)}(a_ia_j+b_ib_j-1)}
\cdot \prod_{j\in V_n} \frac{1}{a_j},\\
\label{eq:L-n+1}
\cL^W_{n+1}(a,b)
= & \prod_{(i\sim j)\in \tilde E_{n+1}} e^{-W_{ij}^{(n+1)}(a_ia_j+b_ib_j-1)}
\cdot \prod_{j\in V_{n+1}} \frac{1}{a_j}.
\end{align}
Since $a_j=1$ for $j\in V_{n+1}\setminus V_n$, one has 
\begin{align}
\prod_{j\in V_n} \frac{1}{a_j}=\prod_{j\in V_{n+1}} \frac{1}{a_j}.
\end{align}
Consider $(i\sim j)\in\tilde E_{n+1}$. \\
\emph{Case $i,j\in V_n$:} Then $(i\sim j)\in\tilde E_n$ and 
$W_{ij}^{(n)}=W_{ij}^{(n+1)}$. Consequently,
$W_{ij}^{(n)}(a_ia_j+b_ib_j-1)=W_{ij}^{(n+1)}(a_ia_j+b_ib_j-1).$\\
{\it Case $i,j\in\tilde V_{n+1}\setminus V_n$:} Then
$[a_i,b_i]=[1,0]=[a_j,b_j]$ and hence $a_ia_j+b_ib_j-1=0$. \\
{\it Case $i\in V_n$ and $j\in\tilde V_{n+1}\setminus V_n$:} 
Then $[a_j,b_j]=[1,0]$. For the given $i\in V_n$, we calculate 
\begin{align}
&\sum_{\substack{j\in\tilde V_{n+1}\setminus V_n:\\ (i\sim j)\in\tilde E_{n+1}}}
W_{ij}^{(n+1)}(a_ia_j+b_ib_j-1)
= \left[ W_{i\delta_{n+1}}^{(n+1)} + \sum_{j\in V_{n+1}\setminus V_n} W_{ij} \right] 
(a_i-1)\\
&= \sum_{j\in V_\infty\setminus V_n} W_{ij} (a_i-1)
= W_{i\delta_n}^{(n)}(a_i-1)
= W_{i\delta_n}^{(n)}(a_ia_{\delta_n}+b_ib_{\delta_n}-1).\nonumber
\end{align}
We conclude that the products over edge sets in \eqref{eq:L-n} and \eqref{eq:L-n+1}
agree. The claim \eqref{eq:Ln-Ln+1equal} follows. This identity holds in particular for 
$(a^2+b^2-1,b)$ in a neighborhood of the origin. As a consequence, by 
analytic continuation, the characteristic function of 
$(\beta^{V_n},\theta^{V_n})$ with respect to $\mu^W_n$ agrees with 
the characteristic function of $(\beta^{V_{n+1}},\theta^{V_{n+1}})|_{V_n}$ with 
respect to $\mu^W_{n+1}$. The claim follows. 
\end{proof}

\medskip\noindent
\begin{proof}[Proof of Theorem~\ref{thm:generating-mg} 
(Generating martingale)]
By the definitions \eqref{eq:def-u-n} and \eqref{eq:def-s-n} of $u^{(n)}$ and 
$s^{(n)}$, it follows that $M^{(n)}_\alpha$ is $\F_n$-measurable. 
For $[a,b]\in\G_{V_{n+1}}$ with $[a_i,b_i]=[1,0]$ for 
all $i\in \tilde V_{n+1}\setminus V_n$, we show 
\begin{align}
\label{eq:expectation-mg2}
\E_{\mu^W_\infty}\Big[ M^{(n+1)}_\alpha 
\prod_{j\in V_n} e^{-(a_j^2+b_j^2-1){\pmb{\beta}_j}-b_j{\pmb{\theta}_j}}\Big]
=\E_{\mu^W_\infty}\Big[ M^{(n)}_\alpha \prod_{j\in V_n} e^{-(a_j^2+b_j^2-1){\pmb{\beta}_j}-b_j{\pmb{\theta}_j}} \Big].
\end{align}
Note that for $j\in V_n$, one has $a_j>0$ and $b_j\in\R$. 
So in particular, we prove the identity \eqref{eq:expectation-mg2}
for $a_j^2+b_j^2-1$ and $b_j$ belonging to a neighborhood of the origin, 
which implies the martingale property for $M^{(n)}_\alpha$. 

We rewrite the claim in terms of expectations with respect to the supersymmetric
sigma model on finite graphs. Let 
\begin{align}
\tilde M^{(n)}_\alpha:\Omega_{V_n}\to\R, \quad 
(u,s)\mapsto e^{\sk{\alpha^{(n)},e^u(1+is)}}.
\end{align}
Using the definition of the variables $\pmb{\beta}$ and $\pmb{\theta}$, 
the identity \eqref{eq:expectation-mg2} is equivalent to 
\begin{align}
\label{eq:expectation-mg3}
& \E_{\mu^W_{n+1}}\Big[ \tilde M^{(n+1)}_\alpha \prod_{j\in V_n} 
e^{-(a_j^2+b_j^2-1)\beta_j^{V_{n+1}}(u)-b_j\theta_j^{V_{n+1}}(u,s)} \Big] \nonumber\\
= & \E_{\mu^W_n}\Big[ \tilde M^{(n)}_\alpha
\prod_{j\in V_n} e^{-(a_j^2+b_j^2-1)\beta_j^{V_n}(u)-b_j\theta_j^{V_n}(u,s)} \Big].
\end{align}
Since $a_j^2+b_j^2-1=0=b_j$ for $j\in\tilde V_{n+1}\setminus V_n$, 
we rewrite the left-hand side of \eqref{eq:expectation-mg3} 
using Theorem~\ref{thm:scaling} as follows: 
\begin{align}
\lhs\eqref{eq:expectation-mg3} 
= & \E_{\mu^W_{n+1}}\Big[ \tilde M^{(n+1)}_\alpha \prod_{j\in V_{n+1}} 
e^{-(a_j^2+b_j^2-1)\beta_j^{V_{n+1}}(u)-b_j\theta_j^{V_{n+1}}(u,s)} \Big] \cr
=& \cL^W_{n+1}(a,b)
\E_{\mu_{n+1}^{W^a}}\left[ \tilde M^{(n+1)}_\alpha\circ \s_{[a,b]} \right],
\label{eq:intermediate-step-mg}
\end{align}
where the last expectation is taken with respect to the supersymmetric sigma 
model on the graph $\tilde G_{n+1}$ with the rescaled weights 
$a_ia_jW_{ij}^{(n+1)}$. We calculate
\begin{align}
 \tilde M^{(n+1)}_\alpha\circ \s_{[a,b]} 
= &\exp\left(\sk{\alpha^{(n+1)},e^{u+\log a}(1+i(s-e^{-u-\log a}b))}\right)
\nonumber \\
= & e^{\sk{a\alpha^{(n+1)},e^u(1+is)}}e^{-\sk{\alpha^{(n+1)},ib}}.
\label{eq:M-n-circ-S}
\end{align}
Note that $\sk{\alpha^{(n+1)},ib}$ does not depend on $u$ or $s$. Consequently, 
inserting the last expression into \eqref{eq:intermediate-step-mg}, we obtain
\begin{align}
\label{eq:lhs-exp-mg1}
\lhs\eqref{eq:expectation-mg3} 
=& \cL^W_{n+1}(a,b) e^{-\sk{\alpha^{(n+1)},ib}}
\E_{\mu_{n+1}^{W^a}}\left[ e^{\sk{a\alpha^{(n+1)},e^u(1+is)}} \right].
\end{align}
By Corollary 5.3 in \cite{disertori-merkl-rolles2017}, 
\begin{align}
\E_{\mu_{n+1}^{W^a}}\left[ e^{\sk{a\alpha^{(n+1)},e^u(1+is)}} \right]
= & e^{\sk{a\alpha^{(n+1)},1_{\tilde V}}} 
=e^{\sk{\alpha^{(n+1)},a}} .
\end{align}
We conclude
\begin{align}
\label{eq:lhs-exp-mg2}
\lhs\eqref{eq:expectation-mg3} 
= \cL^W_{n+1}(a,b) e^{\sk{\alpha^{(n+1)},a-ib}}.
\end{align}
The right-hand side of \eqref{eq:expectation-mg3} 
can be obtained from the last expression by replacing $n+1$ by $n$. 
Thus, the claim \eqref{eq:expectation-mg3} can be written as follows 
\begin{align}
\label{eq:final-equivalent-claim}
\cL^W_{n+1}(a,b) e^{\sk{\alpha^{(n+1)},a-ib}}
= & \cL^W_n(a_{V_n},b_{V_n}) e^{\sk{\alpha^{(n)},a-ib}}.
\end{align}
By Lemma~\ref{le:consistency}, $\cL^W_n(a_{V_n},b_{V_n})=\cL^W_{n+1}(a,b)$. 
Furthermore, using $[a_{\delta_{n+1}},b_{\delta_{n+1}}]=[1,0]$, we obtain
\begin{align}
\sk{\alpha^{(n+1)},a-ib}
= &\sum_{j\in V_{n+1}}\alpha_j(a_j-ib_j)+\alpha^{(n+1)}_{\delta_{n+1}}
= \sum_{j\in V_{n+1}}\alpha_j(a_j-ib_j)+ \sum_{j\in V_\infty\setminus V_{n+1}}\alpha_j 
\nonumber\\
=& \sum_{j\in V_n}\alpha_j(a_j-ib_j)+ \sum_{j\in V_\infty\setminus V_n}\alpha_j
=\sk{\alpha^{(n)},a-ib}.
\label{eq:sk-1}
\end{align}
This shows that \eqref{eq:final-equivalent-claim}
holds and finishes the proof of the martingale property. 
\end{proof}

\medskip\noindent
\begin{proof}[Proof of Corollary~\ref{eq:cor-derivatives-are-mgs}]
By Theorem~\ref{thm:generating-mg}, $(M^{(n)}_\alpha)_{n\in\N}$ is a
 martingale for all $\alpha\in (-\infty,0]^{(V_\infty)}$. The martingale 
property is equivalent to 
\begin{align}
\label{eq:expectation-mg4}
\E_{\mu^W_\infty}[ M^{(n+1)}_\alpha 1_A]=\E_{\mu^W_\infty}[ M^{(n)}_\alpha 1_A]
\end{align}
for all $n\in\N_0$ and all events $A\in\F_n$. 
Taking left-sided derivatives at $\alpha=0$, we get 
\begin{align}
& \partial_{\alpha_{j_1}}\ldots\partial_{\alpha_{j_k}}M^{(n)}_\alpha
=\partial_{\alpha_{j_1}}\ldots\partial_{\alpha_{j_k}}e^{\sk{\alpha^{(n)},e^{u^{(n)}}(1+is^{(n)})}}
=M_{j_1,\ldots,j_k}^{(n)}M^{(n)}_\alpha, \cr
& \partial_{\alpha_{j_1}}\ldots\partial_{\alpha_{j_k}}M^{(n)}_\alpha|_{\alpha=0}
=M_{j_1,\ldots,j_k}^{(n)}. 
\end{align}
Since $|\partial_{\alpha_{j_1}}\ldots\partial_{\alpha_{j_k}}M^{(n)}_\alpha|
\le |M_{j_1,\ldots,j_k}^{(n)}|$
for all $\alpha\in(-\infty,0]^{(V_\infty)}$, we can interchange expectation and 
differentiation at $\alpha=0$ in \eqref{eq:expectation-mg4}. This yields 
the martingale property for $M_{j_1,\ldots,j_k}^{(n)}$. 
\end{proof}

The following are special cases of Corollary~\ref{eq:cor-derivatives-are-mgs}.
\begin{itemize}
\item Since $M_j^{(n)}=e^{u_j^{(n)}}\left(1+is_j^{(n)}\right)$, we know that 
\begin{align}
\left(s_j^{(n)}e^{u_j^{(n)}}\right)_{n\in\N}
\end{align}
is a martingale.
\item One has 
$M_{j,l}^{(n)}
=e^{u_j^{(n)}+u_l^{(n)}}\left(1-s_j^{(n)}s_l^{(n)} 
+ i\left(s_j^{(n)}+s_l^{(n)}\right) \right)$.
Hence, 
\begin{align}
\left(e^{u_j^{(n)}+u_l^{(n)}}\left(1-s_j^{(n)}s_l^{(n)}\right)\right)_{n\in\N}
\quad\text{and}\quad
\left(e^{u_j^{(n)}+u_l^{(n)}}\left(s_j^{(n)}+s_l^{(n)}\right)\right)_{n\in\N}
\end{align}
are martingales. For $j=l$, this yields the martingales
\begin{align}
\left(e^{2u_j^{(n)}}\left(1-\big(s_j^{(n)}\big)^2\right)\right)_{n\in\N}
\quad\text{and}\quad
\left(2s_j^{(n)}e^{2u_j^{(n)}}\right)_{n\in\N}.
\end{align}
\item One has 
\begin{align}
M_{j,l,m}^{(n)} 
= & e^{u_j^{(n)}+u_l^{(n)}+u_m^{(n)}}
\left(1-s_j^{(n)}s_l^{(n)}-s_j^{(n)}s_m^{(n)}-s_l^{(n)}s_m^{(n)}  \right.\nonumber\\
& \left. + i\left(s_j^{(n)}+s_l^{(n)}+s_m^{(n)}-s_j^{(n)}s_l^{(n)}s_m^{(n)}\right)\right).
\end{align}
Hence, the following are martingales:
\begin{align}
& \left(e^{u_j^{(n)}+u_l^{(n)}+u_m^{(n)}}
\left(1-s_j^{(n)}s_l^{(n)}-s_j^{(n)}s_m^{(n)}-s_l^{(n)}s_m^{(n)}\right)\right)_{n\in\N}, \\
&  \left(e^{u_j^{(n)}+u_l^{(n)}+u_m^{(n)}} 
\left(s_j^{(n)}+s_l^{(n)}+s_m^{(n)}-s_j^{(n)}s_l^{(n)}s_m^{(n)}\right)\right)_{n\in\N}, \\
& \left(e^{3u_j^{(n)}}\left(1-3\big(s_j^{(n)}\big)^2\right)\right)_{n\in\N}, \quad 
\left(e^{3u_j^{(n)}}\left(3s_j^{(n)}-\big(s_j^{(n)}\big)^3\right)\right)_{n\in\N}.
\end{align}
\end{itemize}

\section{Extension to Grassmann variables}
\label{sect:3}

We consider now the full supersymmetric $H^{2|2}$ model,  
studied in \cite{disertori-spencer-zirnbauer2010}, including 
Grassmann variables. We start with some preliminaries in Sections 
\ref{subsec:grassmann-algebra} and \ref{subsec:superfns}. In the 
remaining part, we extend the scaling transformation, 
the Laplace transform, and the martingales introduced in the 
previous sections to include Grassmann variables. 

\subsection{Grassmann algebras}
\label{subsec:grassmann-algebra}

Let $\V$ be a finite dimensional $\R$-vector space. Let
\begin{align}
\dach\V:=\bigoplus_{n=0}^{\dim\V}\dach^n\V, \quad
\dach\V_\even:=\bigoplus_{\substack{0\le n\le\dim\V\\ n\text{ even}}}\dach^n\V, \quad
\dach\V_\odd:=\bigoplus_{\substack{0\le n\le\dim\V\\ n\text{ odd}}}\dach^n\V
\end{align}
be the Grassmann algebra generated by it, its even and its odd 
subspace, respectively. In particular, $\R=\dach^0\V\subseteq\dach\V$ and 
$\V=\dach^1\V\subseteq\dach\V$.
The Grassmann product is bilinear and associative. Moreover, for 
all $w,w'\in\dach\V_\odd$ it is anticommutative: $ww'=-w'w$. In particular, 
$w^2=0$. 
Let $\body:\dach\V\to\dach^0\V=\R$ be the projection to the 
0th component and $\soul:\dach\V\to\bigoplus_{n=1}^{\dim\V}\dach^n\V$, 
$\soul(w)=w-\body(w)$, denote the projection to the nilpotent part.  
The subset of positive even elements is defined by 
\begin{align}
\dach\V_\even^+=\{a\in\dach\V_\even:\body(a)>0\}.
\end{align}
As a generalization of \eqref{eq:def-matrix-ab}, 
for $a\in\dach\V_\even^+,b\in\dach\V_\even,\overline w,w\in\dach\V_\odd$, 
we set 
\begin{align}
[a,b,\overline w,w]:=
\left(\begin{array}{cccc}
a & b & \overline w & w \\
0 & 1 & 0 & 0\\
0 & 0 & 1 & 0\\
0 & 0 & 0 & 1
\end{array}\right).
\end{align}
The set of matrices, cf.\ \eqref{eq:def-G}, 
\begin{align}
\label{eq:def-G-grassmann}
\bG(\V):=\{[a,b,\overline w,w]:a\in\dach\V_\even^+,b\in\dach\V_\even,
\overline w,w\in\dach\V_\odd\}
\end{align}
endowed with matrix multiplication forms a group, non-Abelian except
in trivial cases, 
with the neutral element $[1,0,0,0]$. In other words,
\begin{align}
& [a,b,\overline w,w]\cdot[a',b',\overline w',w']
=[aa',b+ab',\overline w+a\overline w',w+aw'],\\
& [a,b,\overline w,w]^{-1}=[a^{-1},-ba^{-1},-\overline w a^{-1},-wa^{-1}]; 
\end{align}
cf.\ \eqref{eq:group-op} and \eqref{eq:inverse}. Note that $a^{-1}$ 
is well-defined because $\body(a)>0$. 

We take again a finite graph $\tilde G=(\tilde V,\tilde E)$ 
with $\tilde V=V\cup\{\delta\}$ as 
in Subsection \ref{subsec:susy-model}. 
We define the cartesian power of the group $\bG(\V)$ with one component 
pinned to the neutral element: 
\begin{align}
\label{eq:def-G-V-grassmann}
\bG(\V)_V:= \{ [a,b,\overline w,w]:=
([a_i,b_i,\overline w_i,w_i])_{i\in \tilde V}\in\bG(\V)^{\tilde V} :\; 
[a_\delta,b_\delta,\overline w_\delta,w_\delta]=[1,0,0,0]\}.
\end{align}

\subsection{Superfunctions and superexpectation}
\label{subsec:superfns}

Let 
\begin{align}
\label{eq:def-A-V}
\A(\V)=\A_V(\V)=C^\infty(\Omega_V,\dach\V)
=C^\infty(\Omega_V,\R)\otimes\dach\V
\end{align} 
be the Grassmann algebra over $\V$ with coefficients
being smooth real-valued functions $f\in C^\infty(\Omega_V,\R)$,
$(u,s)\mapsto f(u,s)$. Elements of $\A(\V)$ are called superfunctions. 

Assume that the vector space $\V$ has a basis 
$(\overline\psi_i,\psi_i)_{i\in V}$. Moreover, we set 
\begin{align}
\overline\psi_\delta=\psi_\delta=0.
\end{align}
Then, $\overline\psi_i,\psi_i\in\V\subseteq\dach\V_\odd$, $i\in\tilde V$, 
implies $\psi_i\overline\psi_j=-\overline\psi_j\psi_i$,
$\psi_i\psi_j=-\psi_j\psi_i$, and
$\overline\psi_i\overline\psi_j=-\overline\psi_j\overline\psi_i$ 
for all $i,j\in\tilde V$. 
To describe a superfunction in $\A(\V)$, the following abbreviations are useful: 
\begin{align}
\I_V=\{(i_1,\ldots,i_n)\in V^n:n\in\N_0, i_1<\ldots<i_n\}
\end{align}
with respect to some fixed linear order $<$ of the vertex set $V$. 
For $I=(i_1,\ldots,i_n)\in\I_V$, we set 
\begin{align}
\label{eq:def-psi-I}
\psi_I=\psi_{i_1}\cdots\psi_{i_n}
\end{align}
and similarly for $\overline\psi_I$. By convention, 
$\overline\psi_\emptyset=\psi_\emptyset=1$. 
Thus, a superfunction $f\in\A(\V)$ can be uniquely written as 
\begin{align}
\label{eq:repr-f}
f(u,s,\overline\psi,\psi)=\sum_{I,J\in\I_V}f_{IJ}(u,s)\overline\psi_I\psi_J
\end{align}
with coefficients $f_{IJ}\in C^\infty(\Omega_V,\R)$. 
Here $f_{\emptyset\emptyset}$ is the body of $f$ and 
$f-f_{\emptyset\emptyset}$ its nilpotent part. An element $f\in\A(\V)$ 
is even if $f_{IJ}=0$ whenever $|I|+|J|$ is odd; $f$ is odd if 
$f_{IJ}=0$ whenever $|I|+|J|$ is even. Let 
$\A(\V)_\even=C^\infty(\Omega_V,\dach\V_\even)$ and 
$\A(\V)_\odd=C^\infty(\Omega_V,\dach\V_\odd)$ 
denote the set of even and odd elements of $\A(\V)$, respectively, and  
let $\A(\V)_\even^+=\{f\in\A(\V)_\even:\body(f)>0\}$. 
Smooth functions (like $\exp$) of elements in $\A(\V)_\even$ are understood as 
power series in the nilpotent part. 

In analogy to the parameter dependent $W^a$ in formula 
\eqref{eq:def-rescaled-weights}
we will consider a further generalization of the supersymmetric sigma model
$H^{2|2}$ from \cite{disertori-spencer-zirnbauer2010}
involving parameters that depend on Grassmann variables. 
Our parameters belong to another Grassmann algebra 
$\dach\V'$ with another finite-dimensional $\R$-vector space $\V'$. 
Both vector spaces $\V$ and $\V'$ are viewed as subspaces of 
their direct sum $\V''=\V\oplus\V'$. The corresponding Grassmann algebras
are related by $\dach\V''=\dach\V\otimes\dach\V'$ with the Grassmann 
product extended to be anticommuting. In particular, 
$\dach\V=\dach\V\otimes\R\subseteq\dach\V''$ and 
$\dach\V'=\R\otimes\dach\V'\subseteq\dach\V''$. 

We will consider superfunctions $f\in\A(\V'')$. Each such 
function can be represented as in \eqref{eq:repr-f} 
with coefficients $f_{IJ}\in\A(\V')$. In the following, we consider 
coupling constants $W_{ij}\in{\dach\V'}^+_\even$ 
for all $(i\sim j)\in\tilde E$
and $W_{ij}=0$ whenever $(i\sim j)\not\in\tilde E$. 
We define the superdensity $\boldsymbol{\rho}^W\in\A(\V'')_\even^+$ by 
\begin{align}
\label{eq:def-bold-rho}
\boldsymbol{\rho}^W(u,s,\overline\psi,\psi)
= & e^{-\frac12\sk{s,A^{W}(u)s}}e^{-\sk{\overline\psi,A^{W}(u)\psi}}
\prod_{(i\sim j)\in \tilde E} e^{-W_{ij}[\cosh(u_i-u_j)-1]}\cr
=& \frac{e^{-\sk{\overline\psi,A^{W}(u)\psi}}}{\det A_{VV}^W(u)} \rho^W(u,s)
\end{align}
with the matrix $A^{W}(u)\in\R^{\tilde V\times\tilde V}$ defined in \eqref{eq:def-A} 
and the density $\rho^W$ defined in \eqref{eq:def-rho}. As Lemma~\ref{le:rho-marginal} 
below shows, $\rho^W$ is the marginal of $\boldsymbol{\rho}^W$. Therefore we use 
the same symbol writing the supersymmetric variant with the corresponding bold symbol. 
This convention will also be used below for other quantities like 
$\zeta$, $\mu^W$, and $\cL^W$. Let 
\begin{align}
\label{eq:def-ell}
\boldsymbol{d\zeta}_V=\boldsymbol{d\zeta}_V[u,s,\overline{\psi},\psi]:=
\prod_{i\in V}\frac{1}{2\pi}\zeta(du_ids_i)\partial_{\overline\psi_i}\partial_{\psi_i}
=\prod_{i\in V}\frac{e^{-u_i}}{2\pi}du_ids_i\partial_{\overline\psi_i}\partial_{\psi_i}
\end{align}
be the supersymmetric reference measure, where we suppress again the Dirac measure
$\delta_{(0,0)}(du_\delta\, ds_\delta)$ in the notation. With these notions the 
supersymmetric sigma model is given by 
\begin{align}
\boldsymbol{\mu}^W(du\, ds\,\partial_{\overline\psi}\,\partial_\psi)
:=& \boldsymbol{d\zeta}_V[u,s,\overline{\psi},\psi]\circ
\boldsymbol{\rho}^W(u,s,\overline\psi,\psi),
\label{eq:def-mu-W-with-grassmann}
\end{align}
where the symbol $\circ$ means that the partial derivatives 
$\partial_{\overline\psi}$ and $\partial_\psi$ act not only on the
superdensity $\boldsymbol{\rho}^W(u,s,\overline\psi,\psi)$, but 
also on the test function as follows:
\begin{align}
\int\boldsymbol{d\mu}^W f
=& \int_{\Omega_V} \boldsymbol{d\zeta}_V(\boldsymbol{\rho}^W f)
\end{align}
for any $f\in\A(\V'')$ for which the integral is defined. 
Here we use the convention 
\begin{align}
\partial_{\overline\psi_i}\,\overline\psi_i =\partial_{\psi_i}\,\psi_i = 1. 
\end{align}
Note that the superintegral 
$\int\boldsymbol{d\mu}^W f$ with integrable arguments $f\in\A(\V'')$ takes 
values in $\dach\V'$.

\begin{lemma}
\label{le:rho-marginal}
The probability measure $\mu^W$ defined in \eqref{eq:def-mu-W} is the marginal of 
the supermeasure $\boldsymbol{\mu}^W$ defined in \eqref{eq:def-mu-W-with-grassmann} 
in the following sense. 
In the special case when the weights 
$W_{ij}$ are real-valued and the superfunction $f$ is an ordinary function 
$f=f(u,s)$, 
i.e.\ does not depend on any Grassmann variables, we have the real-valued
integral
\begin{align}
\int\boldsymbol{d\mu}^W f=\int d\mu^W f.
\end{align}
\end{lemma}
\begin{proof}
Since $f$ is an ordinary function, the Grassmann part in $\int\boldsymbol{d\mu}^W f$
is reduced to 
\begin{align}
\prod_{i\in V}\partial_{\overline\psi_i}\partial_{\psi_i}
e^{-\sk{\overline\psi,A^{W}(u)\psi}}=\det A^W_{VV}(u).
\end{align}
Therefore, the definition \eqref{eq:def-bold-rho} of $\boldsymbol{\rho}^W$
yields
\begin{align}
\prod_{i\in V}\partial_{\overline\psi_i}\partial_{\psi_i}
\boldsymbol{\rho}^W(u,s,\overline\psi,\psi)
=\rho^W(u,s).
\end{align}
The result follows. 
\end{proof}

\subsection{Super scaling transformation}

We generalize now 
the definition \eqref{eq:S-in-terms-of-right-operation} of the 
scaling transformation $\s_{[a,b]}:\Omega_V \to\Omega_V$ to the present
setup involving Grassmann parameters. 
Take a superparameter $[a,b,\overline{\chi},\chi]\in\bG(\V')_V$; recall 
that $[a_\delta,b_\delta,\overline\chi_\delta,\chi_\delta]=[1,0,0,0]$ by 
\eqref{eq:def-G-V-grassmann}.
In order to find an analogue to equation 
\eqref{eq:claim-thm-image-measure-integral-version-new}, 
we consider a generalization of the pull-back 
\begin{align}
\s_{[a,b]}^*f:=f\circ\s_{[a,b]}, \quad f:\Omega_V\to\R
\end{align}
to a supertransformation $\sstar_{[a,b,\overline{\chi},\chi]}:\A(\V'')\to\A(\V'')$ defined as follows. Take a general element
\begin{align}
f(u,s,\overline\psi,\psi)=\sum_{I,J\in\I_V}f_{IJ}(u,s)\overline\psi_K\psi_L\in\A(\V'')
\end{align}
with coefficients $f_{IJ}\in\A(\V')$. In the following, for any even 
$u',s'$, we interpret $f_{IJ}(u',s')$ again as power series in the 
nilpotent part of $u'$ and $s'$. We set
\begin{align}
(\sstar_{[a,b,\overline{\chi},\chi]}f)(u,s,\overline\psi,\psi)
=\sum_{I,J\in\I_V}f_{IJ}(u',s')\overline\psi_K'\psi_L'\in\A(\V''),
\end{align}
where the expressions for $u'=u'(u),s'=s'(u,s),\overline\psi'=
\overline\psi'(u,\overline\psi),\psi'=\psi'(u,\psi)$ are given by the following
formula, to be read componentwise
\begin{align}
[e^{-u'},s',\overline\psi',\psi']
=
[e^{-u},s,\overline\psi,\psi]
\cdot
[a,b,\overline{\chi},\chi]^{-1}.
\end{align} 
This means that the explicit expressions for $u'$, $s'$, $\overline\psi'$,
and $\psi'$ are given by
\begin{align}
\label{eq:def-strich}
& u_i'=u_i+\log a_i, \quad
s_i'=s_i-e^{-u_i}b_ia_i^{-1},\\
& {\overline\psi}_i'=\overline\psi_i-e^{-u_i}\overline\chi_ia_i^{-1},\quad
\psi_i'=\psi_i-e^{-u_i}\chi_ia_i^{-1}\nonumber
\end{align}
for all $i\in V$. 
Note that $[e^{-u_\delta'},s_\delta',\overline\psi_\delta',\psi_\delta']=[1,0,0,0]$, 
and that $u_i'$ and $s_i'$ are even superfunctions in $\A(\V'')$. 

Note that $\sstar_\cdot$ is a group operation, i.e.\ for all $v,v'\in\G(\V)_V$, 
\begin{align}
\sstar_{[1,0,0,0]}=\id, \quad 
\sstar_{v\cdot v'}=\sstar_v\sstar_{v'}, \quad
\sstar_{v^{-1}}=(\sstar_v)^{-1}. 
\end{align}

We will need the following transformation formula for the supermeasure
$d\boldsymbol{\zeta}_V$ with respect to $\sstar_\cdot$. 

\begin{lemma}
\label{le:transformation-zeta}
For $v=[a,b,\overline\chi,\chi]\in\G(\V')_V$ and for any compactly 
supported (or sufficiently fast decaying) test superfunction 
$f\in\A(\V'')$, one has
\begin{align}
\int d\boldsymbol{\zeta}_V\sstar_v f 
= \prod_{j\in V}a_j\int d\boldsymbol{\zeta}_V f. 
\end{align}
\end{lemma}
\begin{proof}
Using $(\sstar_v)^{-1}(e^{-u_i})=e^{-(u_i-\log a_i)}$ and using the 
supertransformation formula described in Lemma 
\ref{le:supertransformation-formula} in the appendix, we calculate
\begin{align}
\int d\boldsymbol{\zeta}_V\sstar_v f 
= & (2\pi)^{-|V|}\int \prod_{i\in V}\, du_ids_i\partial_{\overline\psi_i}
\partial_{\psi_i} \left((\sstar_vf)(u,s,\overline\psi,\psi)
\prod_{i\in V}e^{-u_i}\right)\nonumber\\
=& (2\pi)^{-|V|}\int \prod_{i\in V}\, du_ids_i\partial_{\overline\psi_i}
\partial_{\psi_i} \sstar_v\left(f(u,s,\overline\psi,\psi)
\prod_{i\in V}e^{-(u_i-\log a_i)}\right)\nonumber\\
=& (2\pi)^{-|V|}\int \prod_{i\in V}\, du_ids_i\partial_{\overline\psi_i}
\partial_{\psi_i} f(u,s,\overline\psi,\psi)
\prod_{i\in V}e^{-(u_i-\log a_i)}. 
\end{align}
The claim follows. 
\end{proof}

\subsection{Grassmann-Laplace transform}
In analogy to the definition \eqref{eq:def-theta} of $\theta^{V,W}$, we define 
odd superfunctions $\overline\phi^{V,W}(u,\overline\psi)$ and 
$\phi^{V,W}(u,\psi)$ by 
\begin{align}
\label{eq:def-phi}
\overline\phi^{V,W}(u,\overline\psi)=e^{-u}_{VV}A^W_{VV}(u)\overline\psi_V, \quad
\phi^{V,W}(u,\psi)=e^{-u}_{VV}A^W_{VV}(u)\psi_V,
\end{align}
cf.\ \eqref{eq:theta-componentwise}. 
Here, the restriction $\psi_V=(\psi_i)_{i\in V}$ should not be confused with 
the product $\psi_I$, $I\in\I_V$, defined in \eqref{eq:def-psi-I}.
Componentwise, we have for $i\in V$
\begin{align}
\overline\phi_i^{V,W}(u,\overline\psi)
=\sum_{j\in\tilde V}W_{ij}e^{u_j}(\overline\psi_i-\overline\psi_j), 
\quad\phi_i^{V,W}(u,\psi) = \sum_{j\in\tilde V}W_{ij}e^{u_j}(\psi_i-\psi_j).
\end{align}
As for $\beta$ and $\theta$, we will drop the dependence on $V$, $W$, or both
if there is no risk of confusion. 

Our goal is to derive a generalization of Theorem~\ref{thm:scaling}
including Grassmann variables. In the following, we abbreviate for 
$[a,b,\overline\chi,\chi]\in\bG(\V')_V$
\begin{align}
\label{eq:def-varpi}
\varpi^V=\varpi^{V,W}=(\beta^V,\theta^V,\overline\phi^V,\phi^V), \quad
\pi^V_{[a,b,\overline\chi,\chi]}=(a^2+b^2+2\overline\chi\chi-1,b,\overline\chi,\chi)_V, 
\end{align}
which fulfill $\varpi^V,\pi^V_{[a,b,\overline\chi,\chi]}
\in\left(\A(\V'')_\even\times\A(\V'')_\even\times\A(\V'')_\odd\times\A(\V'')_\odd\right)^V$.
We use the following generalization of the Euclidean scalar product: 
\begin{align}
\sk{\pi^V_{[a,b,\overline\chi,\chi]},\varpi^V}
=\langle(a^2+b^2+2\overline\chi\chi-1)_V,\beta^W\rangle
+\langle b_V,\theta^W\rangle
+\langle\overline\chi_V,\phi^W\rangle
+\langle\overline\phi^W,\chi_V\rangle.
\end{align}
Note the reversed order of factors in the last product, which causes a 
sign change due to anticommutativity. 

\begin{theorem}
\label{thm:scaling-grassmann}
For $[a,b,\overline\chi,\chi]\in\bG(\V')_V$, the 
joint Grassmann-Laplace transform of $\beta^W$, $\theta^W$, $\phi^W$, and 
$\overline\phi^W$ is given by 
\begin{align}
\label{eq:claim-thm-grassmann-Laplace}
& \int \boldsymbol{d\mu}^W 
e^{-\sk{\pi^V_{[a,b,\overline\chi,\chi]},\varpi^V}}
=  \boldsymbol{\mathcal{L}}^W(a,b,\overline\chi,\chi)
\end{align}
with the constant
\begin{align}
\label{eq:def-L-grassmann}
\boldsymbol{\mathcal{L}}^W(a,b,\overline\chi,\chi)
= & \prod_{(i\sim j)\in \tilde E} 
e^{-W_{ij}(a_ia_j+b_ib_j+\overline\chi_i\chi_j+\overline\chi_j\chi_i-1)}
\cdot \prod_{j\in V}\frac{1}{a_j}\in\dach\V_\even'.
\end{align}
Moreover, for every compactly supported 
(or sufficiently fast decaying\footnote{A sufficient condition is 
given above \eqref{eq:int-cart-hor}.}
in $u$ and $s$) test superfunction $f\in\A(\V'')$ it holds 
\begin{align}
\label{eq:claim-thm-image-measure-grassmann}
\int \boldsymbol{d\mu}^W f
e^{-\sk{\pi^V_{[a,b,\overline\chi,\chi]},\varpi^V}}
= & \boldsymbol{\mathcal{L}}^W(a,b,\overline\chi,\chi)
\int\boldsymbol{d\mu}^{W^a}\sstar_{[a,b,\overline\chi,\chi]}f,
\end{align}
where $W^a=(W^a_{ij}:=a_ia_jW_{ij})_{i,j\in\tilde V}$ 
with $W^a_{ij}\in\dach\V'_\even$. 
\end{theorem}

Note that equation \eqref{eq:claim-thm-image-measure-grassmann} is the 
analogue of 
\eqref{eq:claim-thm-image-measure-integral-version-new}. 
We remark that in the special case $b=0$, $\overline\chi=0=\chi$, 
which was already treated in Theorem 2.1 in \cite{disertori-merkl-rolles2017},
$a^2+b^2+2\overline\chi\chi-1$ just reduces to $a^2-1$, which was called 
$\lambda$ in the citation. If we want the Laplace parameters 
$a^2+b^2+2\overline\chi\chi-1$ and $b$ to be real-valued, 
this enforces the parameters $a$ \emph{not} to be real-valued but to 
take values in 
the even part of a Grassmann algebra. This is why we have to allow 
Grassmann algebra-valued weights $W^a_{ij}\in\dach\V'_\even$ rather than 
only real-valued weights.

\medskip\noindent
\begin{proof}[Proof of Theorem~\ref{thm:scaling-grassmann}]
We abbreviate again $v=[a,b,\overline\chi,\chi]$.
Using Lemma~\ref{le:transformation-zeta}, we obtain 
\begin{align}
\int\boldsymbol{d\mu}^{W^a}\,\sstar_vf
=& \int \boldsymbol{d\zeta}_V\, (\boldsymbol{\rho}^{W^a}
\sstar_vf)
= \int \boldsymbol{d\zeta}_V\,(\sstar_v((\sstar_{v^{-1}}
\boldsymbol{\rho}^{W^a})\, f))\nonumber\\
=& \prod_{j\in V}a_j \int \boldsymbol{d\zeta}_V\,
((\sstar_{v^{-1}}\boldsymbol{\rho}^{W^a})\, f).
\label{eq:image-grassmann-integral-variant-new}
\end{align}
Note that 
\begin{align}
(\sstar_{v^{-1}}f)(u,s,\overline\psi,\psi) = 
f(u-\log a,s+e^{-u}b,\overline\psi+e^{-u}\overline\chi,\psi+e^{-u}\chi).
\end{align}
By Lemma~\ref{le:transformed-A}, one has $A^{W^a}(u-\log a)=A^W(u)$
for $a=(a_i)_{i\in\tilde V}\in(\R^+_0)^{\tilde V}$ with $a_\delta=1$. 
Since the entries of the matrix $A^W(u)$ are smooth functions of 
$W_{ij}e^{u_i+u_j}$, this identity remains true if we replace $a_i$, $i\in V$, by 
even elements of the Grassmann algebra $\dach\V'$ with $\body(a_i)>0$. 
Consequently (cf.\ \eqref{eq:transformed-A}), 
\begin{align}
\label{eq:matrix-A-invariant-under-supertrafo}
\sstar_{v^{-1}} A^{W^a}=A^W.
\end{align}
The relation \eqref{eq:expr1-rho} allows us to rewrite 
$\boldsymbol{\rho}^{W^a}$ as follows: 
\begin{align}
\boldsymbol{\rho}^{W^a}(u,s,\overline\psi,\psi)
= & e^{-\frac12\sk{s,A^{W^a}(u)s}}e^{-\sk{\overline\psi,A^{W^a}(u)\psi}}
e^{-\frac12\sk{e^{-u}_{\tilde V},A^{W^a}(u)e^{-u}_{\tilde V}}}. 
\end{align}
Using \eqref{eq:matrix-A-invariant-under-supertrafo} and 
the expression \eqref{eq:def-H-beta} for $H^W_{\tilde\beta(u)}$, we calculate
\begin{align}
& \sstar_{v^{-1}}(\sk{\overline\psi,A^{W^a}(u)\psi})
= \sk{\overline\psi+e^{-u}\overline\chi,A^W(u)(\psi+e^{-u}\chi)}
\nonumber\\
=& \sk{\overline\psi,A^W(u)\psi}+\sk{\overline\phi^W(u,\overline\psi),\chi_V}
+\sk{\overline\chi_V,\phi^W(u,\psi)}
+\sk{\overline\chi,H^W_{\tilde\beta(u)}\chi}.
\end{align}
As in \eqref{eq:trafo-gauss} and \eqref{eq:trafo-cosh}, we obtain 
\begin{align}
\sstar_{v^{-1}}\sk{s,A^{W^a}(u)s}
=&\sk{s,A^{W}(u)s} + \sk{b,H_{\tilde\beta(u)}^Wb}+2\sk{b_V,\theta^W(u,s)}, \\
\sstar_{v^{-1}}\sk{e^{-u}_{\tilde V},A^{W^a}(u)e^{-u}_{\tilde V}}
= & \sk{a,H_{\tilde\beta(u)}^Wa}.
\end{align}
Combining the above identities and relation \eqref{eq:expr2-rho}, we find 
\begin{align}
\sstar_{v^{-1}}\boldsymbol{\rho}^{W^a}(u,s,\overline\psi,\psi)
= & \boldsymbol{\rho}^W(u,s,\overline\psi,\psi)
e^{-\frac12 \left(\sk{a,H_{\tilde\beta(u)}^Wa}
+\sk{b,H_{\tilde\beta(u)}^Wb}
+2\sk{\overline\chi,H^W_{\tilde\beta(u)}\chi}
-\sk{1_{\tilde V},H^W_{\tilde\beta(u)} 1_{\tilde V}}\right)}\nonumber\\
& \cdot e^{-\sk{b_V,\theta^W(u,s)}}
e^{-\sk{\overline\phi^W(u,\overline\psi),\chi_V}
-\sk{\overline\chi_V,\phi^W(u,\psi)}}.
\label{eq:rho-composed-with-S}
\end{align}
Using $a_\delta^2+b_\delta^2+2\overline\chi_\delta\chi_\delta-1=0$, we 
rewrite the first exponent in the last expression as follows
\begin{align}
& -\frac12\left(\sk{a,H_{\tilde\beta(u)}^Wa}
+\sk{b,H_{\tilde\beta(u)}^Wb}+2\sk{\overline\chi,H^W_{\tilde\beta(u)}\chi}
-\sk{1_{\tilde V},H^W_{\tilde\beta(u)} 1_{\tilde V}}\right) \nonumber\\
= & \sum_{(i\sim j)\in \tilde E} W_{ij}
(a_ia_j+b_ib_j+\overline\chi_i\chi_j+\overline\chi_j\chi_i-1)
-\sum_{i\in V}(a_i^2+b_i^2+2\overline\chi_i\chi_i-1)\beta_i^W.
\end{align}
Substituting this in \eqref{eq:rho-composed-with-S} and the result 
in \eqref{eq:image-grassmann-integral-variant-new}, claim 
\eqref{eq:claim-thm-image-measure-grassmann} follows. 
Formula \eqref{eq:claim-thm-grassmann-Laplace} is the special case of 
\eqref{eq:claim-thm-image-measure-grassmann} for $f$ being the constant $1$. 
\end{proof}

\subsection{Ward identities}

To use symmetries of the supersymmetric sigma model, we consider 
cartesian coordinates $x=(x_i)_{i\in\tilde V}$, $y=(y_i)_{i\in\tilde V}$, 
$z=(z_i)_{i\in\tilde V}$, $\xi=(\xi_i)_{i\in\tilde V}$, and 
$\eta=(\eta_i)_{i\in\tilde V}$ defined by 
\begin{align}
\label{eq:change-of-coordinates1}
& x_i=\sinh u_i-\left(\frac12s_i^2+\overline\psi_i\psi_i\right)e^{u_i}, \quad
y_i=s_ie^{u_i}, \quad
\xi_i=e^{u_i}\overline\psi_i, \quad
\eta_i=e^{u_i}\psi_i, \\
\label{eq:change-of-coordinates2}
& z_i=\sqrt{1+x^2_i+y^2_i+2\xi_i\eta_i}=\cosh u_i+\left(\frac12s_i^2+\overline\psi_i\psi_i\right)e^{u_i}.
\end{align}
In particular, $x_\delta=y_\delta=\xi_\delta=\eta_\delta=0$ and $z_\delta=1$.  
Let 
\begin{align}
\label{eq:def-scart}
{\mathcal{S}}_\cart(x,y,\xi,\eta)=
-\sum_{(i\sim j)\in \tilde E}W_{ij}(-1-x_ix_j-y_iy_j+z_iz_j-\xi_i\eta_j+\eta_i\xi_j)
\end{align}
and define 
\begin{align}
\int \boldsymbol{d\mu}^W_\cart f 
:=& \int \prod_{i\in V}
\frac{dx_idy_i}{2\pi}\, \partial_{\xi_i}\partial_{\eta_i}
\left(\prod_{i\in V}\frac{1}{z_i}\cdot e^{{\mathcal{S}}_\cart(x,y,\xi,\eta)} 
f(x,y,\xi,\eta)\right)
\end{align}
for any compactly supported or sufficiently fast decaying test function $f$. 

Let $\V_\cart$ denote the $\R$-vector space with basis $(\xi_i,\eta_i)_{i\in V}$. 
Let $\mathbb{S}_\susy(\Omega_V, \xi,\eta)$ denote the space of superfunctions 
of the form 
\begin{align}
\label{eq:expansion-f-cart}
\begin{array}{ll}
f_\cart : & \Omega_V \to \mathcal{A}(\V_\cart)\\
    & (x,y)\mapsto f_\cart(x,y,\xi,\eta)= \sum_{I,J\in\I_V} f_{IJ}(x,y)\xi_I\eta_J, 
\end{array}
\end{align}
where the coefficients $f_{IJ}$ are Schwartz functions and 
\begin{align}
\xi_I=\prod_{i\in I}\xi_i,\quad \eta_J=\prod_{j\in J}\eta_j.
\end{align}
After doing the change of coordinates given in \eqref{eq:change-of-coordinates1},
we obtain the test function in horospherical coordinates 
$f_\hor : \Omega_V \to \mathcal{A}(\V)$, 
\begin{align}
(u,s)\mapsto  & f_\hor(u,s,\overline\psi,\psi) \nonumber\\
 =&  f_\cart\left(x(u,s,\overline{\psi},\psi),
y(u,s,\overline{\psi},\psi),\xi(u,s,\overline{\psi},\psi),\eta(u,s,\overline{\psi},\psi)\right). 
\end{align}
These notions can be directly extended to superfunctions involving parameters
that depend on Grassmann variables by considering
$f_\cart,{\mathcal{S}}_\cart : \Omega_V \to \mathcal{A}(\V_\cart)\otimes\dach\V'.$
Lemma 5.1 of \cite{disertori-merkl-rolles2017} implies that for any superfunction 
$f_\cart(x,y,\xi,\eta)$ 
with the property
$e^{{\mathcal{S}}_\cart}f_\cart\in\mathbb{S}_\susy(\Omega_V,\xi,\eta)\otimes\dach\V'$,
one has
\begin{align}
\label{eq:int-cart-hor}
\int\boldsymbol{d\mu}^W_\cart f_\cart=\int\boldsymbol{d\mu}^W f_\hor.
\end{align}

\begin{lemma}[Ward identities]
\label{le:rot-sym-grassmann}
Let $f:\C\to\C$ be a holomorphic function and 
$\tau=(\tau_i)_{i\in V}\in(\dach\V_\odd')^V$.
If $f(\sk{\alpha,x+z+iy}+\sk{\tau,\xi+i\eta})e^{{\mathcal{S}}_\cart}\in
\mathbb{S}_\susy(\Omega_V,\xi,\eta)\otimes\dach\V'$, 
then the following identity holds 
\begin{align}
\label{eq:ward-identity-cartesian}
\int\boldsymbol{d\mu}^W_\cart f(\sk{\alpha,x+z+iy}+\sk{\tau,\xi+i\eta})
=f(\sk{\alpha,1}). 
\end{align}
\end{lemma}
\begin{proof}
Let $\varphi\in\R$. We define $\xi^\varphi=(\xi^\varphi_j)_{j\in\tilde V}$, 
$\eta^\varphi=(\eta^\varphi_j)_{j\in\tilde V}$ by 
\begin{align}
\begin{pmatrix}\xi^\varphi_j\\ \eta^\varphi_j\end{pmatrix}
=\begin{pmatrix}\cos\varphi & \sin\varphi\\
-\sin\varphi & \cos\varphi\end{pmatrix}
\begin{pmatrix}\xi_j\\ \eta_j\end{pmatrix}. 
\end{align}
Note that ${\mathcal{S}}_\cart(x,y,\xi,\eta)
={\mathcal{S}}_\cart(x,y,\xi^\varphi,\eta^\varphi)$. Furthermore, 
the supertransformation 
$(x,y,\xi,\eta)\mapsto(x,y,\xi^\varphi,\eta^\varphi)$ has 
super Jacobian 1 and hence leaves the 
reference supermeasure $dx\, dy\, \partial_\xi\partial_\eta$ invariant. 
The assumption $f(\sk{\alpha,x+z+iy}+\sk{\tau,\xi+i\eta})e^{{\mathcal{S}}_\cart}\in
\mathbb{S}_\susy(\Omega_V,\xi,\eta)\otimes\dach\V'$ assures that all expectations 
in the following calculations exist and are finite and justifies that we 
can exchange the order of integration in \eqref{eq:average-over-zero-2pi}, 
below. It follows 
\begin{align}
\lhs\eqref{eq:ward-identity-cartesian}
= & \int\boldsymbol{d\mu}^W_\cart f(\sk{\alpha,x+z+iy}+\sk{\tau,\xi^\varphi+i\eta^\varphi})
\nonumber\\
= & \int\boldsymbol{d\mu}^W_\cart f(\sk{\alpha,x+z+iy}+e^{-i\varphi}\sk{\tau,\xi+i\eta}).
\end{align}
Consequently, 
\begin{align}
\lhs\eqref{eq:ward-identity-cartesian}
= & \frac{1}{2\pi}\int_0^{2\pi}
\int\boldsymbol{d\mu}^W_\cart f(\sk{\alpha,x+z+iy}+e^{-i\varphi}\sk{\tau,\xi+i\eta})
\, d\varphi \nonumber\\
= & \int\boldsymbol{d\mu}^W_\cart \frac{1}{2\pi}\int_0^{2\pi}
f(\sk{\alpha,x+z+iy}+e^{-i\varphi}\sk{\tau,\xi+i\eta})\, d\varphi.
\label{eq:average-over-zero-2pi}
\end{align}
Note that 
\begin{align}
g(r):=\frac{1}{2\pi}\int_0^{2\pi}
f(\sk{\alpha,x+z+iy}+e^{-i\varphi}r)\, d\varphi
-f(\sk{\alpha,x+z+iy})
\end{align}
is an analytic superfunction of $r\in\dach\V_\even'$, which vanishes for
all $r\in\R$ by the mean value theorem for holomorphic functions.
Consequently, using that $g(r)$ for $r\in\dach\V_\even'$ is defined as
a Taylor series in the nilpotent part of $\R$, we obtain 
$g(r)=0$ for all $r\in\dach\V_\even'$. This yields 
\begin{align}
\lhs\eqref{eq:ward-identity-cartesian}
= & \int\boldsymbol{d\mu}^W_\cart f(\sk{\alpha,x+z+iy}).
\end{align}
The claim \eqref{eq:ward-identity-cartesian} follows from 
Lemma 5.2 of \cite{disertori-merkl-rolles2017}. 
\end{proof}

\begin{corollary}[Ward identity for $\exp$]
\label{co:expectation-generating-fnc}
For all $\alpha\in(-\infty,0]^{\tilde V}$ and 
$\tau=(\tau_i)_{i\in V}\in(\dach\V_\odd')^V$, one has 
\begin{align}
\label{eq:expectation-generating-fnc}
\int\boldsymbol{d\mu}^W e^{\sk{\alpha, e^u(1+is)}+\sk{\tau,e^u(\overline\psi+i\psi)}}
&=e^{\sk{\alpha, 1}},
\end{align}
using the abbreviation $e^u(1+is)=(e^{u_j}(1+is_j))_{j\in\tilde V}$.
\end{corollary}
\begin{proof}
We apply Lemma~\ref{le:rot-sym-grassmann} to the function $f=\exp$. 
Note that since $\body (x_j+z_j)=\body (e^{u_j})>0$ and $\alpha_j\le 0$
the assumption 
$e^{\sk{\alpha,x+z+iy}+\sk{\tau,\xi+i\eta}}e^{{\mathcal{S}}_\cart}\in
\mathbb{S}_\susy(\Omega_V,\xi,\eta)\otimes\dach\V'$  is satisfied. Using 
\eqref{eq:change-of-coordinates1} and \eqref{eq:change-of-coordinates2}, 
we find $x_j+z_j+iy_j=e^{u_j}(1+is_j)$ and 
$\xi_j+i\eta_j=e^{u_j}(\overline\psi_j+i\psi_j)$ for $j\in\tilde V$. 
This proves the claim. 
\end{proof}

\subsection{Susy martingales}

Consider an infinite graph $ G_\infty=(V_\infty,E_\infty)$. 
As described before \eqref{eq:def-weights-W-n-vertices}, we approximate 
this infinite graph by finite graphs with wired boundary conditions 
$\tilde G_n=(\tilde V_n=V_n\cup\{\delta_n\},\tilde E_n)$ with $V_n\uparrow V_\infty$.
Let $\V_\infty$ be a vector space with a basis denoted by 
$(\overline\psi_i,\psi_i)_{i\in V_\infty}$. 
Let $\V_n\subseteq\V_\infty$ be the subspace generated by 
$(\overline\psi_i,\psi_i)_{i\in V_n}$.  
We set $\overline\psi_{\delta_n}=\psi_{\delta_n}=0$. 
Let $\pi_n:\Omega_{V_{n+1}}\to\Omega_{V_n}$ be the projection 
$((u_i,s_i)_{i\in V_{n+1}},(u_{\delta_{n+1}},s_{\delta_{n+1}})=(0,0))
\mapsto ((u_i,s_i)_{i\in V_n},(u_{\delta_n},s_{\delta_n})=(0,0))$. 
Identifying $f\in\A_{V_n}(\V_n)$ (cf.\ \eqref{eq:def-A-V})
with $f\circ\pi_n\in\A_{V_{n+1}}(\V_{n+1})$, 
we view $\A_{V_n}(\V_n)$ as a subset of $\A_{V_{n+1}}(\V_{n+1})$. 

In order to have Grassmann parameters available, we consider another 
vector space $\V_\infty'$ together with a filtration of finite-dimensional
subspaces $\V_1'\subseteq\V_2'\subseteq\V_3'\subseteq\ldots$, 
$\bigcup_{n=1}^\infty\V_n'=\V_\infty'$. For $i,j\in V_\infty$, we take weights 
$W_{ij}=W_{ji}\in(\dach\V_\infty')_\even$ such that 
$W_{ij}\in(\dach\V_n')^+_\even$ whenever $i\sim j$ is an edge in $\tilde G_n$
for some $n$ 
and $W_{ij}=0$ whenever $i$ and $j$ are not connected by an edge in the 
infinite graph $ G_\infty$.
The edges of $\tilde G_n$ are given the weights $W_{ij}^{(n)}$ defined as in 
\eqref{eq:def-weights-W-n-vertices} and \eqref{eq:def-weights-W-n-delta}. 
Let $\boldsymbol{\mu}^W_n$ denote the 
supersymmetric sigma model with Grassmann variables defined in 
\eqref{eq:def-mu-W-with-grassmann} for the graph $\tilde G_n$ with weights 
$W^{(n)}_{ij}$. 

Let $n\in\N$. Recall the definition \eqref{eq:def-varpi} of 
$\varpi^{V_n}$ and $\pi^{V_n}_{[a,b,\overline\chi,\chi]}$ 
for $[a,b,\overline\chi,\chi]\in\G(\V_n')_{V_n}$. We consider 
the joint Grassmann-Laplace transform 
\begin{align}
\label{eq:grassmann-Laplace-trafo-for-G-n}
\boldsymbol{\mathcal{L}}_n^W(a,b,\overline\chi,\chi)
= \int \boldsymbol{d\mu}^W_n 
e^{-\sk{\pi^{V_n}_{[a,b,\overline\chi,\chi]},\varpi^{V_n}}}.
\end{align}

\paragraph{Test functions.}
Following the discussion above eq. \eqref{eq:int-cart-hor} we will
consider the space $\T_{n}$ of test functions $f\in \A_{V_n}(\V_n)\otimes\dach\V'_{n}$
such that
 $e^{{\mathcal{S}}_\cart}f_\cart\in\mathbb{S}_\susy(\Omega_{{V_{n}}},\xi,\eta)\otimes\dach\V'_{n}$.

\paragraph{Functions of $\beta,\theta,\overline\phi,\phi$.}
Let $\U_n$ be a vector space with basis $(\overline\phi_i,\phi_i)_{i\in V_n}$. 
In analogy to the definition \eqref{eq:def-A-V} of $\A(\V)$, we denote by 
$\B_{V_n}(\U_n)=C^\infty(\R^{V_n}\times\R^{V_n},\dach\U_n)$ the Grassmann algebra over 
$\U_n$ where the coefficients are given by smooth real-valued functions 
$f_{IJ}\in C^\infty(\R^{V_n}\times\R^{V_n},\R)$,
$(\beta,\theta)\mapsto f_{IJ}(\beta,\theta)$. If we insert 
the functions $\beta=\beta^{V_n}(u)$, $\theta=\theta^{V_n}(u,s)$, 
$\overline\phi=\overline\phi^{V_n}(u,\overline\psi)$, and 
$\phi=\phi^{V_n}(u,\psi)$, 
cf.\ formulas \eqref{eq:beta-definition}, \eqref{eq:def-theta}, 
and \eqref{eq:def-phi}, in the representation 
\begin{align}
\label{eq:repr-f-betatheta}
f(\beta,\theta,\overline\phi,\phi)
=\sum_{I,J\in\I_{V_n}}f_{IJ}(\beta,\theta)\overline\phi_I\phi_J\in\B_{V_n}(\U_n), 
\end{align}
the superfunction in horospherical coordinates can be written as 
\begin{align}
f_{\hor} (u,s,\overline\psi,\psi)= f(\varpi^{V_n} (u,s,\overline\psi,\psi))
=\sum_{I,J\in\I_{V_n}}\tilde f_{IJ}(u,s)\overline\psi_I\psi_J. 
\end{align}
Again, these definitions extend directly to functions involving
Grassmann-dependent parameters $\B_{V_n}(\U_n)\otimes\dach\V'_{n}.$

\begin{lemma}[Consistency]\label{le:consistency-grassmann}\mbox{}\\
For $n\in\N$ and $[a,b,\overline\chi,\chi]\in\G(\V_{n+1}')_{V_{n+1}}$
with $[a_i,b_i,\overline\chi_i,\chi_i]=[1,0,0,0]$ for all 
$i\in\tilde V_{n+1}\setminus V_n$, one has 
\begin{align}
\label{eq:L-n-L-n+1}
\boldsymbol{\mathcal{L}}_n^W(a_{V_n},b_{V_n},\overline\chi_{V_n},\chi_{V_n})
=\boldsymbol{\mathcal{L}}_{n+1}^W(a,b,\overline\chi,\chi). 
\end{align}
Consequently, for any superfunction 
$f\in\B_{V_n}(\U_n)\otimes\dach\V'_{n}$ such that $f_{\hor }\in \mathcal{T}_{n}$
one has
\begin{align}
\label{eq:consistency-grassmann}
\int \boldsymbol{d\mu}^W_n f(\varpi^{V_n})
=\int\boldsymbol{d\mu}^W_{n+1} f((\varpi^{V_{n+1}})|_{V_n}).
\end{align}
Informally speaking, this means that the (super-)law of 
$\varpi^{V_n}=(\beta^{V_n},\theta^{V_n},\overline\phi^{V_n},\phi^{V_n})$ 
with respect to $\boldsymbol{\mu}^W_n$ agrees with the (super-)law of 
$\varpi^{V_{n+1}}|_{V_n}
=(\beta^{V_{n+1}},\theta^{V_{n+1}},\overline\phi^{V_{n+1}},\phi^{V_{n+1}})|_{V_n}$ 
with respect to $\boldsymbol{\mu}^W_{n+1}$. 
\end{lemma}
\begin{proof}
Using the expression \eqref{eq:def-L-grassmann} for the Grassmann-Laplace transform, 
the proof of \eqref{eq:L-n-L-n+1} is in complete analogy with the proof of 
Lemma~\ref{le:consistency}, using Theorem~\ref{thm:scaling-grassmann}
as the analogue of Theorem~\ref{thm:scaling} and replacing
expressions of the form $a_ia_j+b_ib_j-1$ originating from
formula \eqref{eq:claim-laplace-trafo}
by expressions 
$a_ia_j+b_ib_j+\overline{\chi}_i\chi_j+\overline{\chi}_j\chi_i-1$,
appearing in formula \eqref{eq:def-L-grassmann}.

To prove  \eqref{eq:consistency-grassmann}, we consider first
the special case 
$f(\varpi^{V_n})=e^{-\sk{\pi^{V_n}_{[a,b,\overline\chi,\chi]},\varpi^{V_n}}}$. We claim $f_{\hor }\in \T_{n}.$
Indeed note that replacing $u$ in $\beta (u)$ with $u=u (x,y,\xi,\eta)$ 
we can write (cf.\ Lemma~\ref{le:expr1-rho}) 
\begin{align}
{\mathcal{S}}_\cart (x,y,\xi,\eta) &=
-\frac{1}{2} \sk{1_{\tilde{V}_{n}},H_{\tilde{\beta}}1_{\tilde{V}_{n}}}-\frac{1}{2}\sk{y,H_{\tilde{\beta}}y}
-\sk{\xi,H_{\tilde{\beta } }\eta},\\
-\sk{\pi^{V_n}_{[a,b,\overline\chi,\chi]},\varpi^{V_n}}&= -{\mathcal{S}}_\cart +C_{W} (a,b,\overline\chi,\chi)
\\
&-
\frac{1}{2}\sk{a,H_{\tilde{\beta}}a} - \frac{1}{2} \sk{(y+b),H_{\tilde{\beta}} (y+b)} -
\sk{(\xi+\overline\chi ), H_{\tilde{\beta}} (\eta +\chi)}\nonumber
\end{align}
where
\begin{equation}
C_{W} (a,b,\overline\chi,\chi):= \sum_{(i\sim j)\in \tilde{E}_{n}} W_{ij}
\left[1-a_{i}a_{j}-b_{i}b_{j}-\overline\chi_{i}\chi_{j}- \overline\chi_{j}\chi_{i}  \right] 
\end{equation}
is a constant in $(\V'_{n})_{\even}.$ Letting $c:=\min\{ \body (a_{j}^{2}): j\in \tilde{V}_{n}\}>0$
we have
\begin{equation}
e^{{\mathcal{S}}_\cart}f_\cart (x,y,\xi,\eta) = e^{c{\mathcal{S}}_\cart (x,y+b,\xi+\overline\chi, \eta+\chi)}
e^{F (x,y,\xi,\eta)} e^{C_{W} (a,b,\overline\chi,\chi)}
\end{equation}
where $\body F(x,y,\xi,\eta) \leq 0,$ and all derivatives of $F$ of any order in $x,y,\xi, \eta$
are algebraic functions of these variables without singularities. Hence
$e^{{\mathcal{S}}_\cart}f_\cart\in\mathbb{S}_\susy(\Omega_{{V_{n}}},\xi,\eta)\otimes\dach\V'_{n}.$

For the special case 
$f(\varpi^{V_n})=e^{-\sk{\pi^{V_n}_{[a,b,\overline\chi,\chi]},\varpi^{V_n}}}$ claim
\eqref{eq:consistency-grassmann} reads
\begin{align}
\label{eq:consistency-grassmann2}
\int \boldsymbol{d\mu}^W_n e^{-\sk{\pi^{V_n}_{[a,b,\overline\chi,\chi]},\varpi^{V_n}}}
=\int\boldsymbol{d\mu}^W_{n+1} 
e^{-\sk{\pi^{V_n}_{[a,b,\overline\chi,\chi]},\varpi^{V_{n+1}}|_{V_n}}}.
\end{align}
This formula is just another way of writing equation \eqref{eq:L-n-L-n+1}.
For the remainder of this proof, we consider $c:=a^2+b^2+2\overline\chi\chi-1$, 
$b$, $\overline\chi$, $\chi$ rather than $a$, $b$, $\overline\chi$, $\chi$
as our list of independent variables, viewing $a=\sqrt{c-b^2-2\overline\chi\chi+1}$
as a function of $(c,b,\overline\chi,\chi)$. This makes sense as long as 
$\body(c-b^2)>-1$. We take all iterated Grassmann derivatives of the form 
$\prod_{k=1}^m\partial_{\chi_{i_k}}
\prod_{\overline k=1}^{\overline m}\partial_{\overline\chi_{\overline i_{\overline k}}}$ 
with $i_k,\overline i_{\overline k}\in V_n$
in equation \eqref{eq:consistency-grassmann2}. Afterwards, we set 
$\chi=0$ and $\overline\chi=0$. For $I,J\in\I_{V_n}$, we obtain 
\begin{align}
\int \boldsymbol{d\mu}^W_n \overline\phi^{V_n}_I\phi^{V_n}_J
e^{-\sk{c_{V_n},\beta^{V_n}}-\sk{b_{V_n},\theta^{V_n}}}
=&\int\boldsymbol{d\mu}^W_{n+1} \overline\phi^{V_{n+1}}_I\phi^{V_{n+1}}_J
e^{-\sk{c_{V_n},\beta^{V_{n+1}}|_{V_n}}-\sk{b_{V_n},\theta^{V_{n+1}}|_{V_n}}}
\label{eq:identity-with-monomials}
\end{align}
for any Grassmann monomial $g$. 
Note that the identity \eqref{eq:identity-with-monomials} holds in particular 
for all real $b,c$ in a neighborhood of the origin. 

For a general function assume first 
the weights $W_{ij}$ take only real values. 
Then, $\beta$ and $\theta$ take only real values because the integration
variables $u$ and $s$ take real values. Hence, using the uniqueness
theorem for Laplace transforms and the representation \eqref{eq:repr-f-betatheta}
of the superfunction $f$, the claim \eqref{eq:consistency-grassmann} follows 
under our additional assumption $W_{ij}\in\R$; note that the hypothesis  
$f\in\B_{V_n}(\U_n)\otimes\dach\V'_{n}$ with $f_{\hor }\in \T_{n}$
provides the necessary 
integrability. Because both sides of the claim \eqref{eq:consistency-grassmann} 
are analytic superfunctions in the weights $W_{ij}$, the claim follows also 
in the general case. 
\end{proof}

We remark that in the above proof, it is essential to allow the 
scaling parameters $a$ to take values in the even part of a Grassmann algebra
rather than taking only real values, because we have written 
$a=\sqrt{c-b^2-2\overline\chi\chi+1}$ with real $c$ and $b$ and 
Grassmann variables $\overline\chi$ and $\chi$. 

For $\alpha\in(-\infty,0]^{(V_\infty)}$ we use again the definition of 
$\alpha^{(n)}$ given in formula \eqref{eq:def-a-n}. On the contrary, 
given $\tau=(\tau_i)_{i\in V_\infty}$ such that 
$\tau_i\in(\dach\V_n')_\odd$ for all $n\in\N$ and $i\in V_n$,
we denote by $\tau^{(n)}$ the restriction 
of $\tau$ to $V_n$. Note that $\dach\V_n'\subseteq\dach\V_{n+1}'$.

The following theorem is an extension of the martingale property stated in 
Theorem~\ref{thm:generating-mg}.

\begin{theorem}
\label{thm:generating-mg-with-grassmann}
For $n\in\N$, $\alpha\in(-\infty,0]^{(V_\infty)}$, and 
$\tau=(\tau_i)_{i\in V_\infty}$ as above, let
\begin{align}
\label{eq:def-M-n-theta-with-grassmann}
M^{(n)}_{\alpha,\tau}=
M^{(n)}_{\alpha,\tau}(u,s,\overline\psi,\psi)=
e^{\sk{\alpha^{(n)},e^u(1+is)}+\sk{\tau_{V_n},e^u(\overline\psi+i\psi)}}.
\end{align}
For any  test superfunction $g\in\B_{V_n}(\U_n)\otimes\dach\V'_{n}$
with $g_{\hor }\in\T_{n}$, 
one has
\begin{align}
\int \boldsymbol{d\mu}_{n+1}^W M^{(n+1)}_{\alpha,\tau}g(\varpi^{V_{n+1}}|_{V_n})
=\int \boldsymbol{d\mu}_n^W M^{(n)}_{\alpha,\tau}g(\varpi^{V_n}).
\label{eq:claim-generating-mg-test-fn}
\end{align}
\end{theorem}
Note that in \eqref{eq:def-M-n-theta-with-grassmann} we need a definition for 
$\alpha_{\delta_n}$ because $e^{u_{\delta_n}}(1+is_{\delta_n})=1$. In contrast 
to this, $e^{u_{\delta_n}}(\overline\psi_{\delta_n}+i\psi_{\delta_n})=0$, 
hence no definition of $\tau_{\delta_n}$ is needed. 

\medskip\noindent
\begin{proof}[Proof of Theorem~\ref{thm:generating-mg-with-grassmann}]
The proof is in complete analogy to the proof of Theorem 
\ref{thm:generating-mg}, with an extended set of variables. 

We consider first the special case 
$g(\varpi^{V_n})=e^{-\sk{\pi^{V_n}_{[a,b,\overline\chi,\chi]},\varpi^{V_n}}}$
with $[a,b,\overline\chi,\chi]\in\bG(\V_n')_{V_n}$. Note 
that with this choice $g_{\hor }\in\T_{n}$. Now, set 
$[a_i,b_i,\overline\chi_i,\chi_i]=[1,0,0,0]$ for 
$i\in\tilde V_{n+1}\setminus V_n$. The fact 
$\sk{\pi^{V_n}_{[a,b,\overline\chi,\chi]},\varpi^{V_{n+1}}|_{V_n}}=
\sk{\pi^{V_{n+1}}_{[a,b,\overline\chi,\chi]},\varpi^{V_{n+1}}}$ and 
equation \eqref{eq:claim-thm-image-measure-grassmann} 
from Theorem~\ref{thm:scaling-grassmann} yield
\begin{align}
& \int \boldsymbol{d\mu}^W_{n+1} M^{(n+1)}_{\alpha,\tau}
e^{-\sk{\pi^{V_n}_{[a,b,\overline\chi,\chi]},\varpi^{V_{n+1}}|_{V_n}}}
=\int \boldsymbol{d\mu}^W_{n+1} M^{(n+1)}_{\alpha,\tau}
e^{-\sk{\pi^{V_{n+1}}_{[a,b,\overline\chi,\chi]},\varpi^{V_{n+1}}}} \nonumber\\
=& \boldsymbol{\mathcal{L}}_{n+1}^W(a,b,\overline\chi,\chi)
\int\boldsymbol{d\mu}_{n+1}^{W^a}\sstar_{[a,b,\overline\chi,\chi]}M^{(n+1)}_{\alpha,\tau}. 
\label{eq:int-susy-mg-intermediate}
\end{align}
The following calculation is analogous to formula \eqref{eq:M-n-circ-S}:
\begin{align}
\sstar_{[a,b,\overline\chi,\chi]}M^{(n+1)}_{\alpha,\tau}
=&\exp\left(\sk{\alpha^{(n+1)},e^{u+\log a}(1+i(s-e^{-u-\log a}b))}\right)\cdot\nonumber\\
&\exp\left(\sk{\tau^{(n+1)},e^{u+\log a}(\overline\psi-e^{-u-\log a}\overline\chi 
+i(\psi-e^{-u-\log a}\chi))}\right)\nonumber\\
= & e^{\sk{a\alpha^{(n+1)},e^u(1+is)}+\sk{a\tau^{(n+1)},e^u(\overline\psi+i\psi)}}
e^{-\sk{\alpha^{(n+1)},ib}-\sk{\tau^{(n+1)},\overline\chi+i\chi}}.
\label{eq:S-star-M}
\end{align}
Inserting this in \eqref{eq:int-susy-mg-intermediate} and using 
the Ward identity from Corollary~\ref{co:expectation-generating-fnc}, 
we obtain the following analog of the calculation from formula 
\eqref{eq:lhs-exp-mg1} to \eqref{eq:lhs-exp-mg2}: 
\begin{align}
&\int \boldsymbol{d\mu}^W_{n+1} M^{(n+1)}_{\alpha,\tau}
e^{-\sk{\pi^{V_{n+1}}_{[a,b,\overline\chi,\chi]},\varpi^{V_{n+1}}}} 
\nonumber\\
=&\boldsymbol{\mathcal{L}}_{n+1}^W(a,b,\overline\chi,\chi)
e^{-\sk{\alpha^{(n+1)},ib}-\sk{\tau^{(n+1)},\overline\chi+i\chi}}
\int\boldsymbol{d\mu}_{n+1}^{W^a}
e^{\sk{a\alpha^{(n+1)},e^u(1+is)}+\sk{a\tau^{(n+1)},e^u(\overline\psi+i\psi)}}
\nonumber\\
=& \boldsymbol{\mathcal{L}}_{n+1}^W(a,b,\overline\chi,\chi)
e^{-\sk{\alpha^{(n+1)},ib}-\sk{\tau^{(n+1)},\overline\chi+i\chi}}
e^{\sk{a\alpha^{(n+1)}, 1}}\nonumber\\
=& \boldsymbol{\mathcal{L}}_{n+1}^W(a,b,\overline\chi,\chi)
e^{\sk{\alpha^{(n+1)},a-ib}-\sk{\tau^{(n+1)},\overline\chi+i\chi}}.
\end{align}
In the same way, replacing $n+1$ by $n$ yields 
\begin{align}
\int \boldsymbol{d\mu}^W_n M^{(n)}_{\alpha,\tau} 
e^{-\sk{\pi^{V_n}_{[a,b,\overline\chi,\chi]},\varpi^{V_n}}}
=\boldsymbol{\mathcal{L}}_n^W((a,b,\overline\chi,\chi)_{V_n})
e^{\sk{\alpha^{(n)},a-ib}-\sk{\tau^{(n)},\overline\chi+i\chi}}.
\end{align}
The consistency result from Lemma~\ref{le:consistency-grassmann}
can be written in the form 
$\boldsymbol{\mathcal{L}}_{n+1}^W(a,b,\overline\chi,\chi)
=\boldsymbol{\mathcal{L}}_n^W((a,b,\overline\chi,\chi)_{V_n})$. 
Identity \eqref{eq:sk-1} states 
$\sk{\alpha^{(n+1)},a-ib}=\sk{\alpha^{(n)},a-ib}$. Finally,
using $\overline\chi_j=\chi_j=0$ for all $j\in\tilde V_{n+1}\setminus V_n$,
we obtain 
\begin{align}
\sk{\tau^{(n+1)},\overline\chi+i\chi}
=\sum_{j\in V_{n+1}}\tau^{(n+1)}_j(\overline\chi_j+i\chi_j)
=\sum_{j\in V_n}\tau^{(n)}_j(\overline\chi_j+i\chi_j)
=\sk{\tau^{(n)},\overline\chi+i\chi}.
\end{align}
It follows that 
\begin{align}
& \int \boldsymbol{d\mu}^W_{n+1} M^{(n+1)}_{\alpha,\tau}
e^{-\sk{\pi^{V_n}_{[a,b,\overline\chi,\chi]},\varpi^{V_{n+1}}|_{V_n}}}
=\int \boldsymbol{d\mu}^W_n M^{(n)}_{\alpha,\tau} 
e^{-\sk{\pi^{V_n}_{[a,b,\overline\chi,\chi]},\varpi^{V_n}}}. 
\end{align}
Using the same argument as in the proof of Lemma 
\ref{le:consistency-grassmann}, replacing the supermeasure 
$\boldsymbol{d\mu}^W_k$, $k\in\{n,n+1\}$, by 
$\boldsymbol{d\mu}^W_k M^{(k)}_{\alpha,\tau}$, the claim 
\eqref{eq:claim-generating-mg-test-fn} follows for any superfunction 
$g\in\B_{V_n}(\U_n)\otimes\dach\V'_{n}$ with $g_{\hor }\in\T_{n}.$ 
\end{proof}

\begin{corollary}
\label{eq:cor-derivatives-are-susy-mgs}
For $n,k,m\in\N$ and $j_1,\ldots,j_k,l_1,\ldots,l_m\in V_{n+1}$, let 
\begin{align}
M_{j_1,\ldots,j_k,l_1,\ldots,l_m}^{(n)}
=\prod_{p=1}^k e^{u_{j_p}^{(n)}}(1+is_{j_p}^{(n)})
\prod_{q=1}^m e^{u_{l_q}^{(n)}}(\overline\psi_{l_q}+i\psi_{l_q}^{(n)}). 
\end{align}
For any superfunction $g\in\B_{V_n}(\U_n)\otimes\dach\V'_{n}$ with 
$g_{\hor }\in\T_{n}$, one has
\begin{align}
\int \boldsymbol{d\mu}_{n+1}^W M_{j_1,\ldots,j_k,l_1,\ldots,l_m}^{(n+1)}
g(\varpi^{V_{n+1}}|_{V_n})
=\int \boldsymbol{d\mu}_n^W M_{j_1,\ldots,j_k,l_1,\ldots,l_m}^{(n)}g(\varpi^{V_n}).
\label{eq:claim-cor-mg-test-fn}
\end{align}
The same holds for the real and imaginary part of 
$M_{j_1,\ldots,j_k,l_1,\ldots,l_m}^{(n)}$. 
\end{corollary}
\begin{proof}{}
In analogy to  Corollary~\ref{eq:cor-derivatives-are-mgs} the proof follows directly
from the Taylor expansion of formula \eqref{eq:claim-generating-mg-test-fn}
with respect to $\alpha$ and $\tau.$
\end{proof}

\begin{appendix}
\section{Coordinate transformations for superfunctions}

We abbreviate $\boldsymbol{x}=(u,s,\overline\psi,\psi)
=(u_i,s_i,\overline\psi_i,\psi_i)_{i\in V}$ and 
$d\boldsymbol{x}=\prod_{i\in V}du_i\, ds_i\,\partial_{\overline\psi_i}\,
\partial_{\psi_i}$. 

\begin{lemma}
\label{le:supertransformation-formula}
For $v\in\G(\V')_V$ and any compactly supported (or sufficiently
fast decaying) superfunction $f$, one has
\begin{align}
\int d\boldsymbol{x} \; (\sstar_v f)(\boldsymbol{x})
=
\int d\boldsymbol{x} \; f(\boldsymbol{x}).
\end{align}
\end{lemma}
\begin{proof}
Consider a supermatrix 
\begin{align}
M=\left(\begin{array}{ll}
A & \Sigma\\
\Gamma & B \end{array}\right)
\end{align}
where $A,B$ have even entries, $\Sigma,\Gamma$ have odd entries, and 
$A$ and $B$ are invertible.  Its superdeterminant is defined by 
\begin{align}
\label{eq:def-superdet}
\sdet M = \frac{\det(A-\Sigma B^{-1}\Gamma)}{\det B}.
\end{align}
It plays an 
analogous role in Berezin's supertransformation formula as the ordinary
determinant plays in the classical transformation formula; cf.\ 
Theorem 2.1 in \cite{berezin}. 

For $v=[a,b,\overline\chi,\chi]$, the change of coordinates generating
$\sstar_v$ is given by 
\begin{align}
\boldsymbol{x}'(\boldsymbol{x})=
(u',s',\overline\psi',\psi')
=(u+\log a,s-e^{-u}ba^{-1},\overline\psi-e^{-u}\overline\chi a^{-1},
\psi-e^{-u}\chi a^{-1}).
\end{align} 
This transformation has the super Jacobi matrix given by 
\begin{align}
\frac{\partial\boldsymbol{x}'}{\partial\boldsymbol{x}}=
\left(\begin{array}{llll}
\frac{\partial u'}{\partial u} & \frac{\partial u'}{\partial s} & 
\frac{\partial u'}{\partial\overline\psi } & 
\frac{\partial u'}{\partial\psi } \\
\frac{\partial s'}{\partial u} & \frac{\partial s'}{\partial s} & 
\frac{\partial s'}{\partial\overline\psi } & 
\frac{\partial s'}{\partial\psi } \\
\frac{\partial\overline\psi'}{\partial u} & 
\frac{\partial\overline\psi'}{\partial s} & 
\frac{\partial\overline\psi'}{\partial\overline\psi } & 
\frac{\partial\overline\psi'}{\partial\psi } \\
\frac{\partial\psi'}{\partial u} & 
\frac{\partial\psi'}{\partial s} & 
\frac{\partial\psi'}{\partial\overline\psi } & 
\frac{\partial\psi'}{\partial\psi } 
\end{array}\right)=
\left(\begin{array}{ll}
A & \0 \\
\Gamma & \boldsymbol{1} \end{array}\right)
\end{align}
with 
\begin{align}
A=\begin{pmatrix} 1 & 0 \\
e^{-u}ba^{-1} & 1
\end{pmatrix}, \quad
\boldsymbol{1}=\begin{pmatrix} 1 & 0 \\
0 & 1
\end{pmatrix}, \quad
\Gamma=\begin{pmatrix} e^{-u}\overline\chi a^{-1} & 0 \\
e^{-u}\chi a^{-1} & 0
\end{pmatrix}. 
\end{align}
Here $e^{-u}ba^{-1}$ is the diagonal matrix with the entries 
$e^{-u_i}b_ia^{-1}_i$. This super Jacobi matrix has the superdeterminant
$\sdet \frac{\partial\boldsymbol{x}'}{\partial\boldsymbol{x}}=1$. 
Consequently, the inverse supertransformation has the superdeterminant
$\sdet \frac{\partial\boldsymbol{x}}{\partial\boldsymbol{x}'}=1$, 
as well. We obtain 
\begin{align}
&\int d\boldsymbol{x} \; (\sstar_v f)(\boldsymbol{x})
= \int d\boldsymbol{x} \; f(\boldsymbol{x}'(\boldsymbol{x}))
=
\int d\boldsymbol{x}' \; f(\boldsymbol{x}')\;
\sdet\frac{\partial\boldsymbol{x}}{\partial\boldsymbol{x}'}
=
\int d\boldsymbol{x}' \; f(\boldsymbol{x}').
\end{align}
\end{proof}
\end{appendix}

\paragraph{Acknowledgements.} F.M.\ and S.R.\ would like to thank Pierre 
Tarr\`es for pointing out the usefulness of Lemma 
\ref{le:distribution-theta-given-u}.

\end{document}